\documentclass{amsart}

\usepackage[utf8]{inputenc}
\usepackage[table]{xcolor}
\usepackage{hyperref}
\usepackage{enumerate}
\usepackage{amsmath}
\usepackage{amssymb}
\usepackage{amsthm}
\usepackage{mathtools}
\usepackage{ytableau}
\usepackage{blkarray}
\usepackage[noabbrev,capitalize]{cleveref}
\usepackage{graphicx}
\usepackage[pdf]{graphviz}
\usepackage{tkz-euclide}
\def\lineThickness{.7}
\tikzset{every picture/.append style={scale=.5,
thinLine/.style={line width=\lineThickness pt},
thickLine/.style={line width=2*\lineThickness pt,line join=round}},
fillGrey/.style={fill=black!30},
entry/.style={xshift=5mm,yshift=5mm,font=\footnotesize},
circled/.style={circle,draw,inner sep=0pt,minimum size=1em}
}
\usepackage[font=footnotesize]{caption, subcaption}

\newtheorem{thm}{Theorem}[section]
\newtheorem{conj}[thm]{Conjecture}

\newtheorem{lem}[thm]{Lemma}
\newtheorem{prop}[thm]{Proposition}
\newtheorem{cor}[thm]{Corollary}
\newtheorem{prob}[thm]{Open problem}

\crefname{defn}{Definition}{Definitions}
\crefname{lem}{Lemma}{Lemmata}
\crefname{thm}{Theorem}{Theorems}
\crefname{prop}{Proposition}{Propositions}
\crefname{cor}{Corollary}{Corollaries}
\crefname{fig}{Figure}{Figures}
\crefname{conj}{Conjecture}{Conjectures}
\crefname{sec}{Section}{Sections}

\newcommand{\doi}[1]{\href{http://doi.org/#1}{\texttt{doi.org/#1}}}

\newcommand{\abs}[1]{\left| #1 \right|}
\newcommand{\bop}[1]{\boldsymbol{\operatorname{#1}}}

\newcommand{\ceil}[1]{\left\lceil #1 \right\rceil}
\newcommand{\e}{\bop{e}}
\newcommand{\matd}{\mathrm{d}}
\newcommand{\E}{\mathbb{E}}

\newcommand{\n}{\bop{n}}
\newcommand{\N}{\mathbb{N}}
\newcommand{\norm}[1]{\left\| #1 \right\|_{\infty}}

\renewcommand{\P}{\mathbb{P}}
\newcommand{\pr}{\operatorname{pr}}

\newcommand{\R}{\mathbb{R}}

\newcommand{\s}{\bop{s}}
\renewcommand{\S}{\mathfrak{S}}

\newcommand{\sh}{\operatorname{sh}}
\newcommand{\shf}[1]{f_{#1}^{\operatorname{sh}}}
\newcommand{\shh}[1]{h_{#1}^{\operatorname{sh}}}

\newcommand{\tab}{T}
\newcommand{\T}[1]{\mathcal{T}_{#1}}

\newcommand{\w}{\bop{w}}
\newcommand{\wo}{w_{0}}
\newcommand{\Y}{\mathcal{Y}}
\newcommand{\yd}[1]{#1^{\operatorname{dg}}}
\newcommand{\shyd}[1]{#1^{\operatorname{sh}}}
\newcommand{\Z}{\mathbb{Z}}

\usepackage{chngcntr}
\counterwithin{table}{section}

\begin{document}
\title[On random shifted SYT and 132-avoiding sorting networks]{On random shifted standard Young tableaux and 132-avoiding sorting networks}

\author{Svante Linusson, Samu Potka and Robin Sulzgruber}
\address{Department of Mathematics, KTH Royal Institute of Technology, Stockholm, Sweden.}
\email{\href{mailto:linusson@kth.se}{linusson@kth.se}, \href{mailto:potka@kth.se}{potka@kth.se}, \href{mailto:robinsul@kth.se}{robinsul@kth.se}}
\thanks{The first two authors were supported by the Swedish Research Council, grant 621-2014-4780.}

\begin{abstract}
We study shifted standard Young tableaux (SYT). The limiting surface of uniformly random shifted SYT of staircase shape is determined, with the integers in the SYT as heights. This implies via properties of the Edelman--Greene bijection results about random 132-avoiding sorting networks, including limit shapes for trajectories and intermediate permutations. Moreover, the expected number of adjacencies in SYT is considered. It is shown that on average each row and each column of a shifted SYT of staircase shape contains precisely one adjacency. 
\end{abstract}

\maketitle

\begin{figure}[t]
\begin{subfigure}{.45\textwidth}
\centering
\begin{tikzpicture}
\begin{scope}
\draw[thinLine](3,1)--(4,1)(2,2)--(4,2)(1,3)--(4,3)(1,3)--(1,4)(2,2)--(2,4)(3,1)--(3,4);
\draw[thickLine](0,3)--(1,3)--(1,2)--(2,2)--(2,1)--(3,1)--(3,0)--(4,0)--(4,4)--(0,4)--cycle;
\draw[entry](0,3)node{$1$}(1,3)node{$2$}(2,3)node{$4$}(3,3)node{$5$}(1,2)node{$3$}(2,2)node{$6$}(3,2)node{$7$}(2,1)node{$8$}(3,1)node{$9$}(3,0)node{$10$};
\end{scope}
\end{tikzpicture}
\caption{A shifted SYT of staircase shape.}
\label{shifted_example}
\end{subfigure}
\begin{subfigure}{.45\textwidth}
\centering
\includegraphics[width=.9\textwidth]{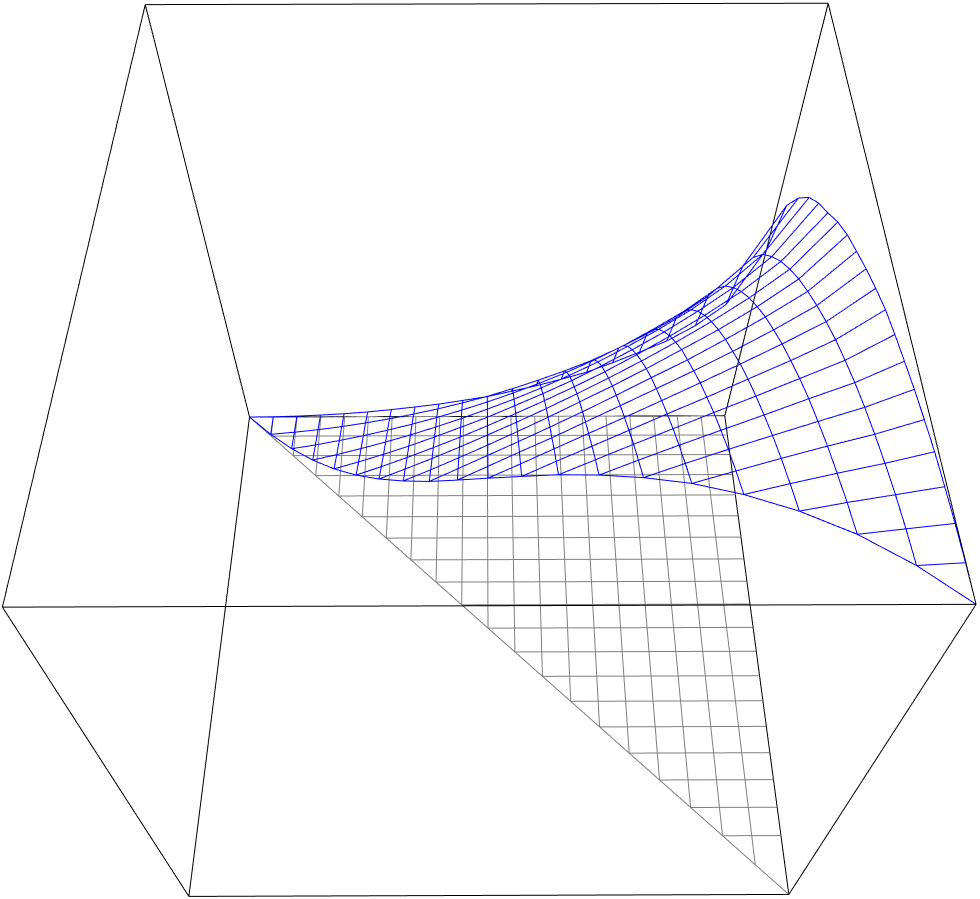}
\bigskip
\caption{The limit shape of uniformly random shifted SYT of staircase shape.}
\label[fig]{shape}
\end{subfigure}
\caption{}
\end{figure}

\section{Introduction}
A shifted standard Young tableau (SYT) of staircase shape is an increasing filling of the shifted diagram of the partition $(n-1,\dots,2,1)$ with the integers $1,2,\dots,\binom{n}{2}$.
See \cref{shifted_example} for an example and \cref{S:background} for the exact definition.
Shifted diagrams and tableaux are important combinatorial objects that appear in various contexts.
In representation theory shifted Young diagrams correspond to projective characters of the symmetric group, and shifted tableaux lend themselves to being studied via RSK-type methods~\cite{Sagan:shifted_tableaux,Stembridge:shifted_tableaux}.
In the theory of partially ordered sets shifted diagrams alongside non-shifted Young diagrams and rooted trees form the three most interesting families of $d$-complete posets, which are in turn connected to fully commutative elements of Coxeter groups~\cite{Stembridge:fully_commutative,Proc:d_complete_poset}.
The most salient property of $d$-complete posets is the fact that their linear extensions (in our case shifted SYT) are enumerated by elegant product formulas. 
Shifted diagrams also appear as order filters in the root poset of type $B_n$, and shifted SYT play an important role in the enumeration of reduced words of elements of the Coxeter group of type $B_n$~\cite{Haiman:mixed_insertion,Kraskiewicz:reduced_decompositions_hyperoctahedral}.
Moreover, as is topical in this paper, shifted SYT are also relevant to the study of certain reduced words in the symmetric group.
Recently Elizalde and Roichman related shifted diagrams and tableaux to unimodal permutations~\cite{ER14}.

\smallskip
The topics of this paper can be divided into three parts.

\smallskip
In \cref{S:shape} we study the surface obtained by thinking of the integers in random SYT as heights.
The study of limit phenomena for partitions and tableaux is an active field of research combining methods from combinatorics, probability theory and analysis.
We refer to~\cite{Romik} for a general survey.
Shifted objects have been treated as well, for example Ivanov~\cite{Ivanov} proves a central limit theorem for the Plancherel measure on shifted diagrams.
In the present paper we determine the limiting surface for uniformly random shifted SYT of staircase shape, see \cref{shape} and \cref{limit_shape}.
The deduction of our results relies on a paper by Pittel and Romik~\cite{PR07} where the limit shape for random rectangular SYT is determined. In fact, we end up with the same variational problem, and the limit surface for shifted staircase SYT is half of the surface for square SYT. This analogy is in part explained by a combinatorial identity \eqref{eq:hL} relating shifted and non-shifted tableaux.
There are very few shapes for which the limit surface has been determined previously. As far as we know the only other case is that of staircase SYT, where again the same limit surface appears, but cut along a different diagonal~\cite{AHRV07}.
Results of this type have applications in other fields of mathematics such as geometric complexity theory~\cite{PS:skew_howe_duality}.


\smallskip
Secondly, we study \emph{132-avoiding sorting networks}, which are by definition reduced words $w_1 \dots w_{\binom n2}$ of the reverse permutation such that $s_{w_1} \cdots s_{w_k}$ is 132-avoiding for any $1 \leq k \leq \binom n2$.
These objects have received considerable recent interest and also appear in different guises, for example as chains of maximum length in the Tamari lattice~\cite{BW97}.
Fishel and Nelson~\cite{FN14} showed that 132-avoiding sorting networks are in bijection with shifted SYT of staircase shape via the Edelman--Greene correspondence.
This has been rediscovered several times~\cite{STWW,DavisSagan,LP18}. They are also in bijection with reduced words of the signed permutation $(-(n-1), -(n-2), \dots, -1)$ via the shift $s_i \mapsto s_{i-1}$ as was remarked in~\cite[Sec.~1.3]{STWW}.
In \cref{S:permutations,S:trajectories} the Edelman--Greene bijection is used to transfer the limit shape of shifted SYT to determine the limit shapes of intermediate permutations (\cref{intermediate_permutations}) and trajectories (\cref{scaled_trajectory}) in random 132-avoiding sorting networks.
These results are motivated by a remarkable paper of Angel, Holroyd, Romik and Vir{\'a}g~\cite{AHRV07} that contains a number of tantalising conjectures about random sorting networks.
See \cref{S:background} for a description of some of the conjectures from~\cite{AHRV07}.
Our results are a parallel to their conjectures restricted to a subclass of random sorting networks. 
Recently a proof of the original conjectures in \cite{AHRV07} was announced by Dauvergne~\cite{D18}.

We remark that the limit surface for shifted SYT of staircase shape contains complete information on the limit surface for SYT of square shape.
This suggests the perhaps less intuitive idea that the relatively small subset of 132-avoiding sorting networks contains a lot of information on random sorting networks in general.  

\smallskip
The third set of results is obtained in \cref{S:adj} and concerns patterns in 132-avoiding sorting networks.
We first observe that adjacencies in a shifted SYT (that is, integers $i$ and $i+1$ in neighbouring cells) translate directly to adjacencies in a 132-avoiding sorting network 
(that is, $j$ and $j+1$ next to each other in the reduced word). \cref{expected_adjacencies} asserts that the expected number of adjacencies in each column and each row in a shifted SYT of staircase shape is exactly 1. The proof uses promotion 
and evacuation techniques very similar to the methods used by Schilling, Thi{\'e}ry, White and Williams~\cite{STWW} to derive results on Yang--Baxter moves (that is, patterns of the form $j(j\pm1)j$) in 132-avoiding sorting networks.
Related results on general sorting networks are due to Reiner~\cite{Reiner} and Tenner~\cite{Tenner}.

\section{Background}\label{S:background}

In this section we fix notation and review some facts about partitions, tableaux and random sorting networks.

For $n\in\N$ let $[n]=\{1,\dots,n\}$.
Throughout this paper we denote $N=\binom{n}{2}$.

\subsection{Partitions and tableaux}\label{S:tableaux}

A \emph{partition} is a weakly decreasing sequence $\lambda=(\lambda_1,\dots,\lambda_n)$ of positive integers.
If a partition is strictly decreasing it is called \emph{strict}.
The sum $\sum\lambda_i$ is called the \emph{size} of the partition $\lambda$ and is denoted by $\abs{\lambda}$.
The number of entries $\lambda_i$ is called the \emph{length} of the partition and is denoted by $\ell(\lambda)$.
Define the \emph{staircase} partition as $\Delta_n=(n-1,\dots,2,1)$.
The \emph{Young diagram} of a partition $\lambda$ is defined as the set
\begin{equation*}
\yd{\lambda}=\{(i,j):i\in[\ell(\lambda)],j\in[\lambda_i]\}.
\end{equation*}
The elements $(i,j)$ are indexed with matrix notation and typically referred to as \emph{cells} of $\lambda$.
The \emph{conjugate partition} $\lambda'$ of the partition $\lambda$ is the partition corresponding to the Young diagram $\{(j,i):(i,j)\in\yd{\lambda}\}$.
Given a strict partition $\lambda$ we also define its \emph{shifted Young diagram} as
\begin{equation*}
\shyd{\lambda}=\{(i,j+i-1):i\in[\ell(\lambda)],j\in[\lambda_i]\}.
\end{equation*}
Thus the shifted Young diagram is obtained from the normal Young diagram by shifting rows to the right.
These definitions are illustrated in \cref{young_diagram}.

\begin{figure}[t]
\centering
\begin{tikzpicture}
\begin{scope}
\draw[thinLine](0,1)--(1,1)(0,2)--(2,2)(0,3)--(4,3)(1,1)--(1,4)(2,2)--(2,4)(3,2)--(3,4)(4,3)--(4,4)(5,3)--(5,4);
\draw[thickLine](0,0)--(1,0)--(1,1)--(2,1)--(2,2)--(4,2)--(4,3)--(6,3)--(6,4)--(0,4)--cycle;
\end{scope}
\begin{scope}[xshift=10cm]
\draw[thinLine](3,1)--(4,1)(2,2)--(4,2)(1,3)--(5,3)(1,3)--(1,4)(2,2)--(2,4)(3,1)--(3,4)(4,2)--(4,4)(5,3)--(5,4);
\draw[thickLine](0,3)--(1,3)--(1,2)--(2,2)--(2,1)--(3,1)--(3,0)--(4,0)--(4,2)--(5,2)--(5,3)--(6,3)--(6,4)--(0,4)--cycle;
\end{scope}
\end{tikzpicture}
\caption{The Young diagram $\yd{\lambda}$ (left) and the shifted Young diagram $\shyd{\lambda}$ (right) of the strict partition $\lambda=(6,4,2,1)$ drawn in \emph{English convention}, that is, $(1,1)$ is the top left cell.
We have $\abs{\lambda}=13$, $\ell(\lambda)=4$ and $\lambda'=(4,3,2,2,1,1)$.}
\label[fig]{young_diagram}
\end{figure}
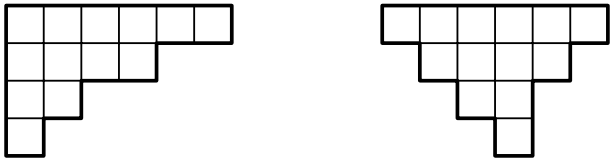

Let $u=(i,j)$ be a cell in the Young diagram of a partition $\lambda$.
The \emph{hook-length} of $u$, denoted by $h_{\lambda}(u)$, is defined as the number of cells in the same row as $u$ that lie to the right of $u$ plus the number of cells in the same column as $u$ that lie weakly below $u$ (that is, including $u$ itself).
Thus, $h_{\lambda}(u)=\lambda_i-j+\lambda_j'-i+1$.
Given a cell $u=(i,j)$ in the shifted Young diagram of a strict partition, we define its \emph{shifted hook-length} $\shh{\lambda}(u)$ as the number of cells in the same row as $u$ that lie to the right of $u$, respectively in the same column and weakly below $u$, plus the number of cells in the $(i+j)$-th row of the shifted Young diagram.
See \cref{hook_lengths}.

\begin{figure}[t]
\centering
\begin{tikzpicture}
\begin{scope}
\draw[thinLine](0,1)--(1,1)(0,2)--(2,2)(0,3)--(4,3)(1,1)--(1,4)(2,2)--(2,4)(3,2)--(3,4)(4,3)--(4,4)(5,3)--(5,4);
\draw[thickLine](0,0)--(1,0)--(1,1)--(2,1)--(2,2)--(4,2)--(4,3)--(6,3)--(6,4)--(0,4)--cycle;
\draw[entry](0,3)node{$9$}(1,3)node{$7$}(2,3)node{$5$}(3,3)node{$4$}(4,3)node{$2$}(5,3)node{$1$}(0,2)node{$6$}(1,2)node{$4$}(2,2)node{$2$}(3,2)node{$1$}(0,1)node{$3$}(1,1)node{$1$}(0,0)node{$1$};
\end{scope}
\begin{scope}[xshift=10cm]
\draw[thinLine](3,1)--(4,1)(2,2)--(4,2)(1,3)--(5,3)(1,3)--(1,4)(2,2)--(2,4)(3,1)--(3,4)(4,2)--(4,4)(5,3)--(5,4);
\draw[thickLine](0,3)--(1,3)--(1,2)--(2,2)--(2,1)--(3,1)--(3,0)--(4,0)--(4,2)--(5,2)--(5,3)--(6,3)--(6,4)--(0,4)--cycle;
\draw[entry](0,3)node{$10$}(1,3)node{$8$}(2,3)node{$7$}(3,3)node{$6$}(4,3)node{$3$}(5,3)node{$1$}(1,2)node{$6$}(2,2)node{$5$}(3,2)node{$4$}(4,2)node{$1$}(2,1)node{$3$}(3,1)node{$2$}(3,0)node{$1$};
\end{scope}
\end{tikzpicture}
\caption{The (shifted) Young diagram of $\lambda=(6,4,2,1)$ with the (shifted) hook-lengths filled into each cell.}
\label[fig]{hook_lengths}
\end{figure}
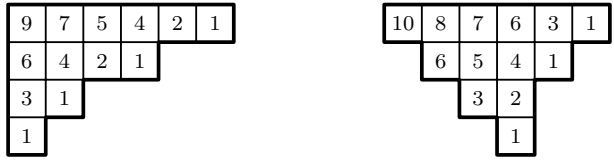

Given a cell $u=(i,j)\in\Z^2$, define the \emph{north}, \emph{east}, \emph{south}, and \emph{west neighbour} of $u$ as
\begin{equation*}
\n u=(i-1,j),\quad
\e u=(i,j+1),\quad
\s u=(i+1,j)\quad\text{and}\quad
\w u=(i,j-1),
\end{equation*}
respectively.
A \emph{tableau} of \emph{shape} $\yd{\lambda}$ is a map $\tab:\yd{\lambda}\to\Z$.
A tableau $\tab$ is called a \emph{standard Young tableau (SYT)} if $\tab:\yd{\lambda}\to[n]$ is a bijection and $\tab(u)<\tab(\e u)$ and $\tab(u)<\tab(\s u)$ whenever the respective cells lie in $\yd{\lambda}$.
Similarly a \emph{shifted standard Young tableau} of \emph{shape} $\shyd{\lambda}$ is a bijection $\tab:\shyd{\lambda}\to[n]$ such that $\tab(u)<\tab(\e u)$ and $\tab(u)<\tab(\s u)$ whenever the respective cells lie in the shifted Young diagram $\shyd{\lambda}$.
Let $\T{n}$ denote the set of shifted SYT of shape $\shyd{\Delta_n}$.
For example, \cref{shifted_example} shows a shifted standard Young diagram of shape $\shyd{\Delta_5}$.

\smallskip
The number of SYT is given be the following nice product formula.

\begin{thm}[Hook-length formula \cite{FRT1954,T52}]\label[thm]{hlf}

Let $\lambda$ be a partition.
Then the number of SYT of shape $\lambda$ is given by 
\begin{equation*}
f_{\lambda}=\frac{\abs{\lambda}!}{\prod_{u\in\yd{\lambda}}h_{\lambda}(u)}.
\end{equation*}
If $\lambda$ is strict then the number of shifted SYT of shape $\shyd{\lambda}$ is given by
\begin{equation*}
\shf{\lambda}=\frac{\abs{\lambda}!}{\prod_{u\in\shyd{\lambda}}\shh{\lambda}(u)}.
\end{equation*}
\end{thm}

The literature offers a variety of proofs of different flavours for \cref{hlf}, for example using hook-walks~\cite{GNW,Sagan:selecting_random_shifted} or by means of jeu de taquin~\cite{NPS,Fischer2002}.

\subsection{Random 132-avoiding sorting networks}

The reader is referred to \cite{K11} for background on pattern avoidance in permutations.

For $i\in[n-1]$ let $s_i=(i,i+1)$ denote the $i$:th \emph{adjacent transposition}.
The \emph{reverse permutation} $\wo\in\S_n$ is defined by $\wo(i)=n-i+1$ for $i \in [n]$.
A \emph{reduced word} of $\wo$ is a word $w=w_1\cdots w_N$ in the alphabet $[n-1]$ such that $w_0=s_{w_1}\cdots s_{w_N}$.

Note that we perform the compositions of transpositions $s_{w_i}$ corresponding to a word $w = w_1 \dots w_m$ from the left. As an example, consider $\S_4$ and the reduced word $1231$. 
Composing  $s_1 s_2 s_3 s_1$ from the left yields the permutation $(3, 2, 4, 1)$. In terms of permutation matrices, we have, for example, 
	\[s_1 = \begin{blockarray}{cccc}
	\mathbf{2} & \mathbf{1} & \mathbf{3} & \mathbf{4} \\
	\begin{block}{(cccc)}
	0 & 1 & 0 & 0 \\
	1 & 0 & 0 & 0 \\
	0 & 0 & 1 & 0 \\
	0 & 0 & 0 & 1 \\
	\end{block}
	\end{blockarray}\quad  s_1 s_2 = \begin{blockarray}{cccc}
	\mathbf{2} & \mathbf{3} & \mathbf{1} & \mathbf{4} \\
	\begin{block}{(cccc)}
	0 & 0 & 1 & 0 \\
	1 & 0 & 0 & 0 \\
	0 & 1 & 0 & 0 \\
	0 & 0 & 0 & 1 \\
	\end{block}
	\end{blockarray}\quad  s_1 s_2 s_3 s_1 = \begin{blockarray}{cccc}
	\mathbf{3} & \mathbf{2} & \mathbf{4} & \mathbf{1} \\ 
	\begin{block}{(cccc)}
	0 & 0 & 0 & 1 \\
	0 & 1 & 0 & 0 \\
	1 & 0 & 0 & 0 \\
	0 & 0 & 1 & 0 \\
	\end{block}
	\end{blockarray}\] where we can see that $s_i$ corresponds to swapping the columns $i$ and $i+1$.

Angel, Holroyd, Romik and Vir\'ag introduced \emph{$n$-element random sorting networks} in~\cite{AHRV07} as the set of reduced words of the reverse permutation $\wo\in\S_n$ equipped with the uniform probability measure.
In the same paper, Angel et al.~pose several striking conjectures about random sorting networks.

Suppose $w = w_1 \dots w_N$ is a sorting network.
Then $w_1 \dots w_k$ defines the \emph{intermediate permutation} $\sigma_k=s_{w_1} \cdots s_{w_k}\in\S_n$ for all $k\in[N]$.
One of the consequences of \cite[Conj.~2]{AHRV07} is that asymptotically the permutation matrices corresponding to the intermediate configurations coming from random sorting networks are supported on a family of ellipses.
In other words, their 1s occur inside an elliptic region of the matrix. In particular, at half-time the permutation matrix is supported on a disc. \cref{matrices} provides an illustration.
\begin{figure}[!htbp]
	\centering
	\includegraphics[width=0.32\textwidth]{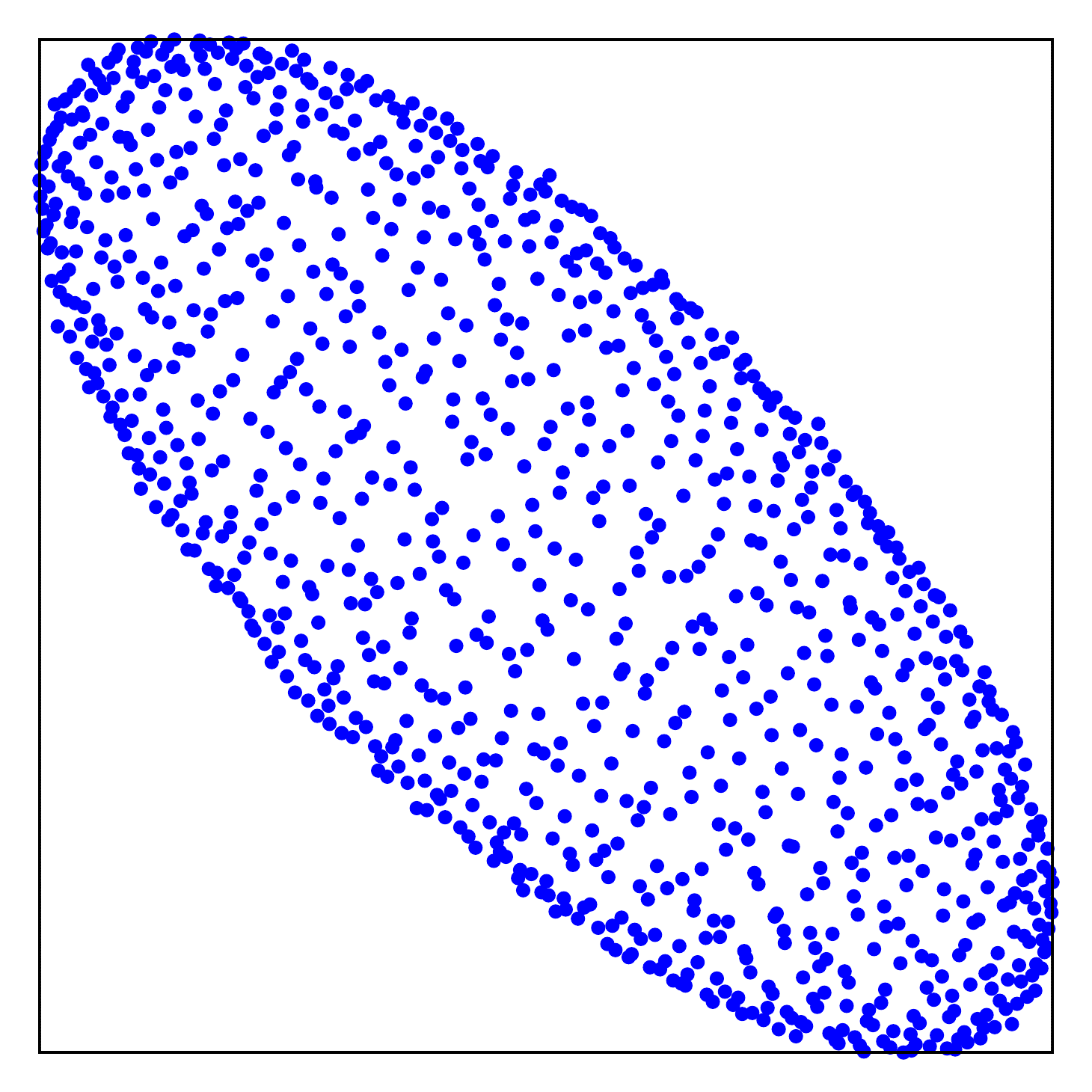}
	\includegraphics[width=0.32\textwidth]{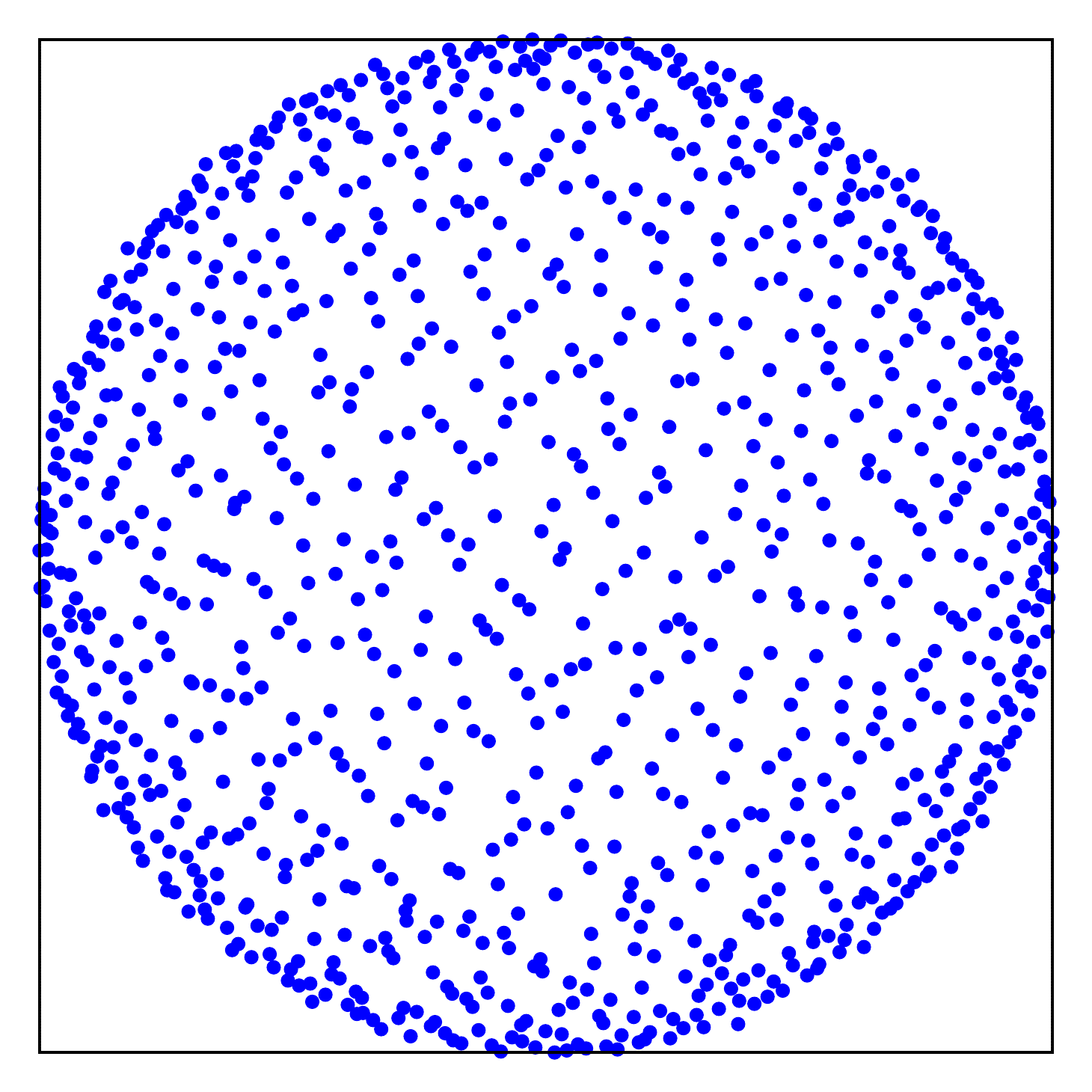}
	\includegraphics[width=0.32\textwidth]{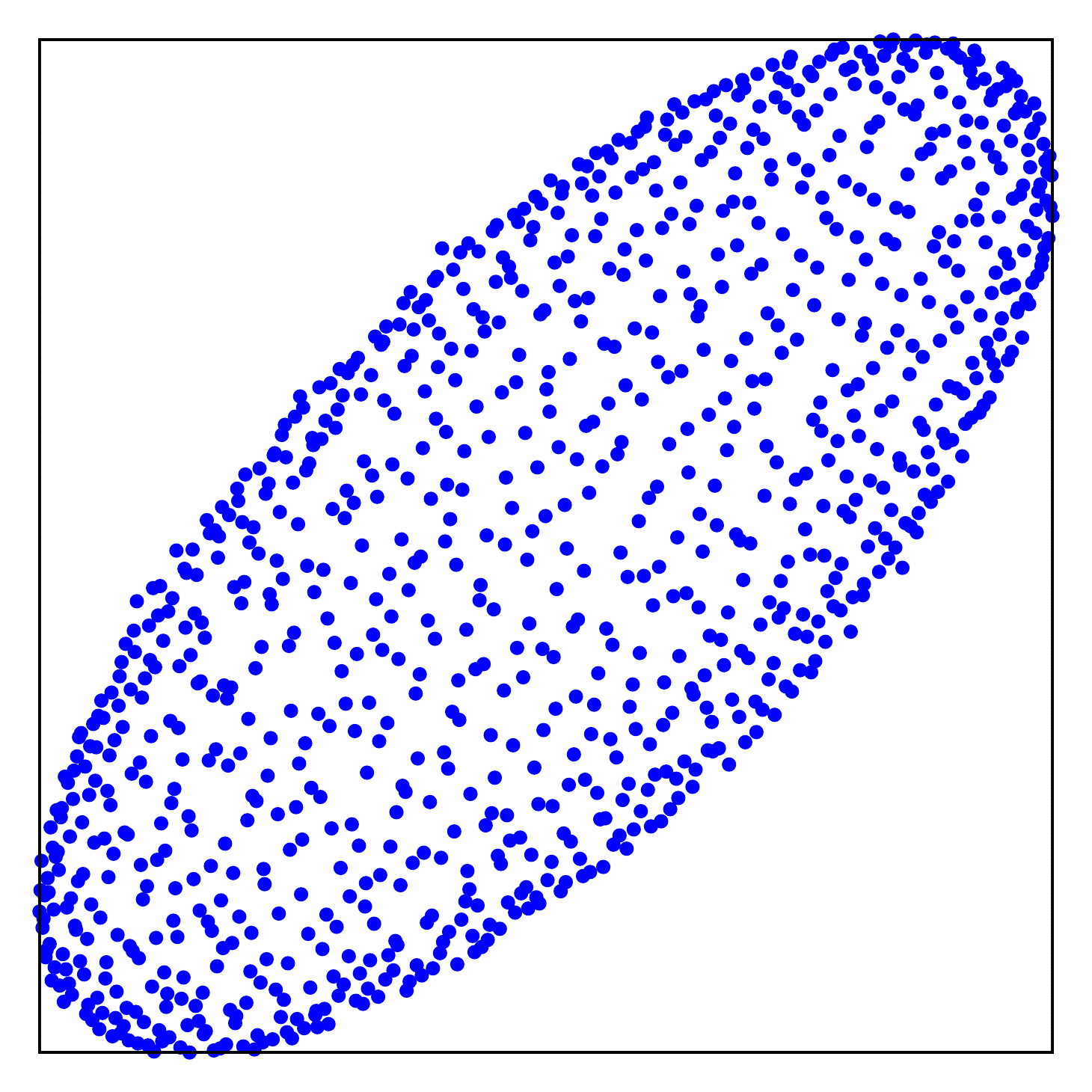}
	\caption{The intermediate permutation matrices of $\sigma_{\lfloor \alpha N \rfloor}$ of a 1000-element random sorting network at times $\alpha = \frac{1}{4}, \frac{1}{2}$ and $\frac{3}{4}$.}
	\label[fig]{matrices}
\end{figure}

For $0\le \alpha\le 1$, \cite[Conj.~1]{AHRV07} states that the scaled trajectories defined by
\begin{equation*}
f_{w, i}(\alpha)
=\frac{2\sigma_{\alpha N}^{-1}(i)}{n} - 1
\end{equation*}
for $\alpha N\in\Z$, and by linear interpolation otherwise, converge to random sine curves. See \cref{trajectories}. \begin{figure}[!htbp]
	\centering
	\includegraphics[width=0.6\textwidth]{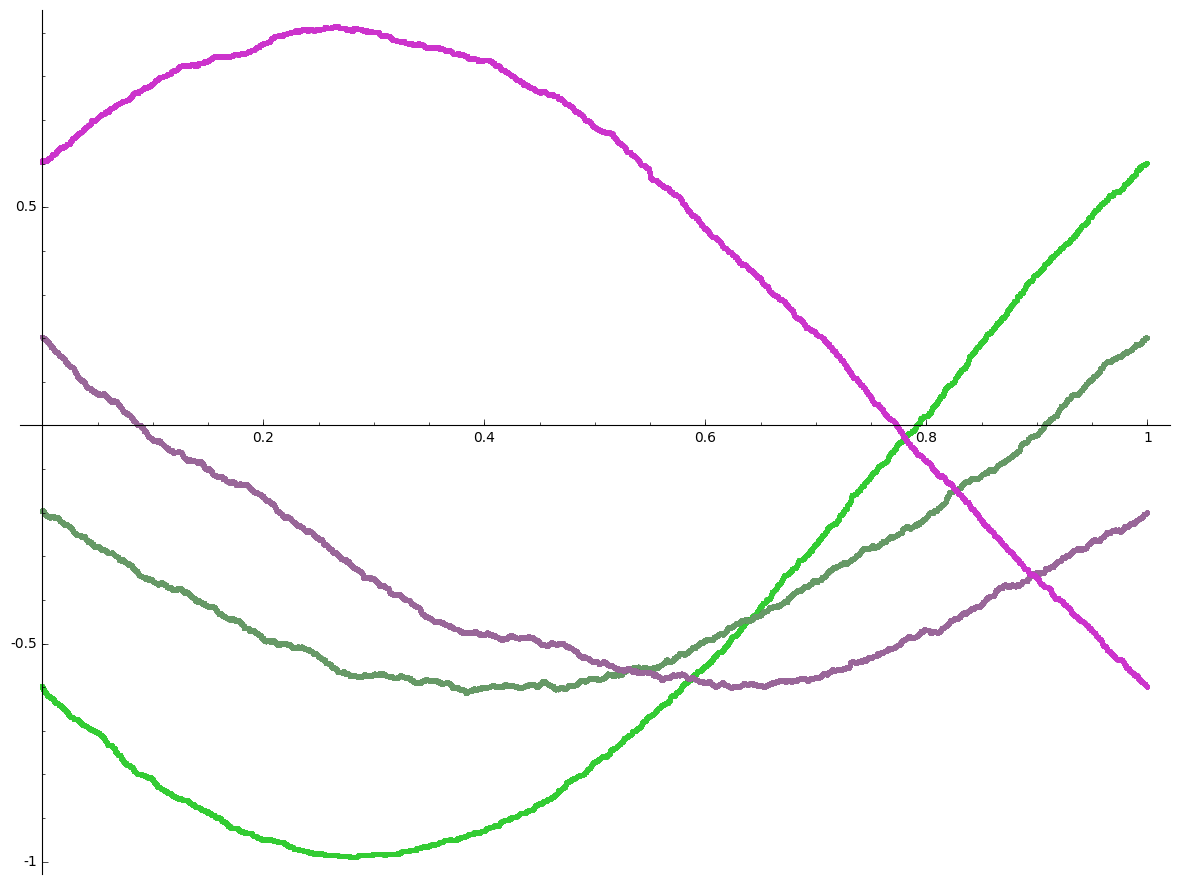}
	\caption{The scaled trajectories of the elements 200, 400, 600 and 800 in a 1000-element random sorting network.}
	\label[fig]{trajectories}
\end{figure}

The \emph{permutahedron} is an embedding of $\S_n$ into $\mathbb{R}^n$ defined by \[\sigma \mapsto (\sigma^{-1}(1), \dots, \sigma^{-1}(n)).\] Every permutation $\sigma \in \S_n$ lies on the sphere
\begin{equation*}
\mathbb{S}^{n-2}
=\Big\{z \in \R^n : \sum_{i = 1}^n z_i = \frac{n(n+1)}{2}
\text{ and }
\sum_{i = 1}^n z_i^2 = \frac{n(n+1)(2n+1)}{6}\Big\}
\,.
\end{equation*}
Random sorting networks correspond to paths on the permutahedron.
The strongest conjecture, \cite[Conj.~3]{AHRV07}, which implies both of the previous ones, states that these paths are close to great circles.
A proof of these conjectures was recently announced by Dauvergne in~\cite{D18}.

\smallskip
This paper considers similar questions restricted to \emph{132-avoiding sorting networks}, that is, those reduced words $w_1 \dots w_N$ of the reverse permutation in $\S_n$ such that $s_{w_1} \cdots s_{w_k}$ is 132-avoiding for all $k\in[N]$. With a  \emph{ random 132-avoiding sorting network}, we will refer to uniform distribution among all such networks of the same length.

The connection between 132-avoiding sorting networks and shifted SYT is the following. Let $w = w_1 \dots w_N$ be a 132-avoiding sorting network. Then, for $k \in [N]$, define a SYT $Q_{w_1 \dots w_k}$ by $Q_{w_1\dots w_k}(\cdot, j) = m$ if $w_m = j$ for all $m \in [k]$. Furthermore, define a shifted SYT $Q_{w_1 \dots w_k}^{\rightarrow}$ by shifting the rows of $Q_{w_1 \dots w_k}$. \cref{eg} shows an example.

\begin{thm}[{\cite[Thm.~3.3 and Thm.~4.6]{FN14}}]
For all $n \in \N$, the map $w \mapsto Q_w^{\rightarrow}$ is a bijection from $n$-element 132-avoiding sorting networks to shifted SYT of shape $\Delta^{\sh}_n$. The map $w \mapsto Q_w$ agrees with the restriction of the Edelman--Greene correspondence to 132-avoiding sorting networks.
\end{thm}
The same bijection was also described in~\cite{LP18}, \cite[Fig.~4]{STWW} in terms of heaps, and \cite[Prop.~5.2]{DavisSagan} in terms of descent sets.


\begin{figure}[b]
\centering
\begin{tikzpicture}
\draw[xshift=-5mm,yshift=15mm](0,0)node[anchor=east]{$w=1213423121$};
\draw[xshift=5mm,yshift=15mm,thickLine,->](0,0)--(1,0);
\begin{scope}[xshift=3cm]
\draw[thinLine](0,1)--(1,1)(0,2)--(2,2)(0,3)--(3,3)(1,1)--(1,4)(2,2)--(2,4)(3,3)--(3,4);
\draw[thickLine](0,0)--(1,0)--(1,1)--(2,1)--(2,2)--(3,2)--(3,3)--(4,3)--(4,4)--(0,4)--cycle;
\draw[entry](0,3)node{$1$}(1,3)node{$2$}(2,3)node{$4$}(3,3)node{$5$}(0,2)node{$3$}(1,2)node{$6$}(2,2)node{$7$}(0,1)node{$8$}(1,1)node{$9$}(0,0)node{$10$};
\end{scope}
\draw[xshift=8cm,yshift=15mm,thickLine,->](0,0)--(1,0);
\begin{scope}[xshift=10cm]
\draw[thinLine](3,1)--(4,1)(2,2)--(4,2)(1,3)--(4,3)(1,3)--(1,4)(2,2)--(2,4)(3,1)--(3,4);
\draw[thickLine](0,3)--(1,3)--(1,2)--(2,2)--(2,1)--(3,1)--(3,0)--(4,0)--(4,4)--(0,4)--cycle;
\draw[entry](0,3)node{$1$}(1,3)node{$2$}(2,3)node{$4$}(3,3)node{$5$}(1,2)node{$3$}(2,2)node{$6$}(3,2)node{$7$}(2,1)node{$8$}(3,1)node{$9$}(3,0)node{$10$};
\end{scope}
\end{tikzpicture}
\caption{The Edelman--Greene correspondence for a 5-element 132-avoiding sorting network.}
\label[fig]{eg}
\end{figure}
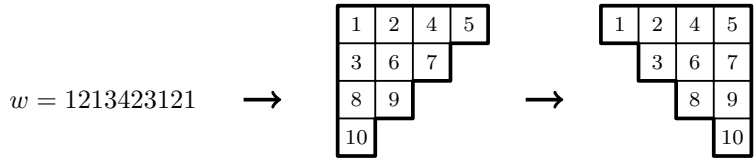

We conclude this section with two more facts about 132-avoiding sorting networks, see \cite{LP18} for proofs. First, reversing a sorting network preserves the property of being 132-avoiding.
\begin{prop}\label{symmetry}
	A reduced word $w_1 \dots w_N$ is a 132-avoiding sorting network if and only if $w_N \dots w_1$ is.
\end{prop}
\begin{proof}
This follows from flipping $Q_w^{\rightarrow}$ about the anti-diagonal and replacing each label $k$ by $N + 1 - k$.
\end{proof}
Second, the set of 132-avoiding and 312-avoiding sorting networks coincide.
\begin{prop}[{\cite[Prop.~3.10]{LP18}}]
	A reduced word $w$ is a 132-avoiding sorting network if and only if $w$ is a 312-avoiding sorting network.
\end{prop}

\section{The limit shape}\label{S:shape}
In this section we derive a limit shape for random shifted SYT of staircase shape.
We may interpret a shifted SYT $\tab\in\T{n}$ as the graph of a function
\begin{equation*}
L_{\tab}:\big\{(x,y)\in\R^2:0<x<y<1\big\}\to\R_{\geq0}
\end{equation*}
by viewing the entries as heights
\begin{equation*}
L_{\tab}(x,y)=\frac{1}{N}\tab\big(\ceil{(n-1)x},\ceil{(n-1)y}).
\end{equation*}
Our main result, \cref{limit_shape}, states that by this choice of scaling, the functions $L_{\tab^{(n)}}$ converge with probability 1 to the surface depicted in \cref{shape}, where each $\tab^{(n)}\in\T{n}$ is chosen uniformly at random.
The proof of \cref{limit_shape} relies heavily on the work of Pittel and Romik~\cite{PR07}.

\smallskip
Recall our conventions from \cref{S:tableaux}. Given a strict partition $\lambda$ define a partition $\Lambda$ by letting its Young diagram equal
\begin{equation*}
\yd{\Lambda}
=\big\{(i,j+1):(i,j)\in\shyd{\lambda}\big\}
\cup\big\{(j,i):(i,j)\in\shyd{\lambda}\big\}.
\end{equation*}
It is easy to see that this really is the Young diagram of a partition.
See \cref{shiftsymm}.
We call $\Lambda$ the \emph{shift-symmetric} partition corresponding to $\lambda$.
The motivation for this definition is the fact that shifted hook-lengths of the cells in $\shyd{\lambda}$ correspond to hook-lengths of cells in $\yd{\Lambda}$. 

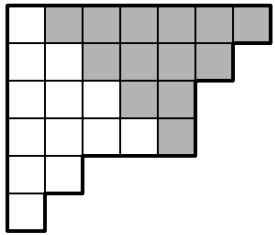
\begin{figure}[t]
\centering
\begin{tikzpicture}
\begin{scope}
\fill[fillGrey,xshift=1cm,yshift=2cm](0,3)--(1,3)--(1,2)--(2,2)--(2,1)--(3,1)--(3,0)--(4,0)--(4,2)--(5,2)--(5,3)--(6,3)--(6,4)--(0,4)--cycle;
\draw[thinLine](0,1)--(1,1)(0,2)--(2,2)(0,3)--(5,3)(0,4)--(5,4)(0,5)--(6,5)(1,1)--(1,6)(2,2)--(2,6)(3,2)--(3,6)(4,2)--(4,6)(5,4)--(5,6)(6,5)--(6,6);
\draw[thickLine](0,0)--(1,0)--(1,1)--(2,1)--(2,2)--(5,2)--(5,4)--(6,4)--(6,5)--(7,5)--(7,6)--(0,6)--cycle;
\end{scope}
\end{tikzpicture}
\caption{The shift-symmetric partition $\Lambda=(7,6,5,5,2,1)$ of the strict partition $\lambda=(6,4,2,1)$.}
\label[fig]{shiftsymm}
\end{figure}

\begin{prop}\label[prop]{fL}
Let $\lambda$ be a strict partition and $\Lambda$ its shift-symmetric partition.
Then
\begin{equation}\label{eq:fL}
f_{\Lambda}=(\shf{\lambda})^2\cdot\binom{2\abs{\lambda}}{\abs{\lambda}}\cdot2^{-\ell(\lambda)}.
\end{equation}
\end{prop}

\begin{proof}
We show that
\begin{equation}\label{eq:hL}
\prod_{u\in\yd{\Lambda}}h_\Lambda(u)
=2^{\ell(\lambda)}\Big(\prod_{u\in\shyd{\lambda}}\shh{\lambda}(u)\Big)^2.
\end{equation}
This is obvious when $\lambda$ is a staircase and $\Lambda$ is a rectangle.
In this case
\begin{equation*}
h_{\Lambda}(i,i)=2\,h_{\Lambda}(i,\ell(\lambda)+1)=2\,\shh{\lambda}(i,\ell(\lambda))
\end{equation*}
as in \cref{diag_hook_lengths} and
\begin{equation*}
h_{\Lambda}(i,j)=h_{\Lambda}(j,i)=\shh{\lambda}(i,j-1)
\end{equation*}
when $i<j\leq\ell(\lambda)$ as in \cref{other_hook_lengths}.
The identity in \eqref{eq:hL} follows inductively as it is easy to verify that it is preserved when a cell $(i,j)$ with $i<j$ is added to $\shyd{\lambda}$.
See \cref{add_cell}.

The claim then follows from the hook-length formula (\cref{hlf}).
\end{proof}

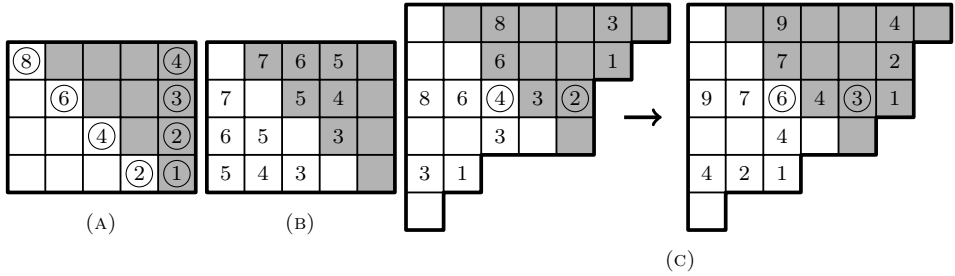
\begin{figure}[t]
\centering
\begin{subfigure}{.2\textwidth}
\begin{tikzpicture}
\begin{scope}
\fill[fillGrey,xshift=1cm](0,3)--(1,3)--(1,2)--(2,2)--(2,1)--(3,1)--(3,0)--(4,0)--(4,4)--(0,4)--cycle;
\draw[thinLine](0,1)--(5,1)(0,2)--(5,2)(0,3)--(5,3)(1,0)--(1,4)(2,0)--(2,4)(3,0)--(3,4)(4,0)--(4,4);
\draw[thickLine](0,0)--(5,0)--(5,4)--(0,4)--cycle;
\draw[entry]
(0,3)node[circled]{$8$}(1,2)node[circled]{$6$}(2,1)node[circled]{$4$}(3,0)node[circled]{$2$}(4,3)node[circled]{$4$}(4,2)node[circled]{$3$}(4,1)node[circled]{$2$}(4,0)node[circled]{$1$};
\end{scope}
\end{tikzpicture}
\caption{}
\label[fig]{diag_hook_lengths}
\end{subfigure}
\begin{subfigure}{.2\textwidth}
\begin{tikzpicture}
\begin{scope}
\fill[fillGrey,xshift=1cm](0,3)--(1,3)--(1,2)--(2,2)--(2,1)--(3,1)--(3,0)--(4,0)--(4,4)--(0,4)--cycle;
\draw[thinLine](0,1)--(5,1)(0,2)--(5,2)(0,3)--(5,3)(1,0)--(1,4)(2,0)--(2,4)(3,0)--(3,4)(4,0)--(4,4);
\draw[thickLine](0,0)--(5,0)--(5,4)--(0,4)--cycle;
\draw[entry](1,3)node{$7$}(2,3)node{$6$}(3,3)node{$5$}(2,2)node{$5$}(3,2)node{$4$}(3,1)node{$3$}(0,2)node{$7$}(0,1)node{$6$}(1,1)node{$5$}(0,0)node{$5$}(1,0)node{$4$}(2,0)node{$3$};
\end{scope}
\end{tikzpicture}
\caption{}
\label[fig]{other_hook_lengths}
\end{subfigure}
\begin{subfigure}{.58\textwidth}
\begin{tikzpicture}
\begin{scope}
\fill[fillGrey,xshift=1cm,yshift=2cm](0,3)--(1,3)--(1,2)--(2,2)--(2,1)--(3,1)--(3,0)--(4,0)--(4,2)--(5,2)--(5,3)--(6,3)--(6,4)--(0,4)--cycle;
\draw[thinLine](0,1)--(1,1)(0,2)--(2,2)(0,3)--(5,3)(0,4)--(5,4)(0,5)--(6,5)(1,1)--(1,6)(2,2)--(2,6)(3,2)--(3,6)(4,2)--(4,6)(5,4)--(5,6)(6,5)--(6,6);
\draw[thickLine](0,0)--(1,0)--(1,1)--(2,1)--(2,2)--(5,2)--(5,4)--(6,4)--(6,5)--(7,5)--(7,6)--(0,6)--cycle;
\draw[entry](5,5)node{$3$}(5,4)node{$1$}(0,3)node{$8$}(1,3)node{$6$}(2,3)node[circled]{$4$}(3,3)node{$3$}(4,3)node[circled]{$2$}(2,5)node{$8$}(2,4)node{$6$}(2,2)node{$3$}(0,1)node{$3$}(1,1)node{$1$};
\end{scope}
\draw[->,thickLine,xshift=58mm](0,3)--(1,3);
\begin{scope}[xshift=75mm]
\fill[fillGrey,xshift=1cm,yshift=2cm](0,3)--(1,3)--(1,2)--(2,2)--(2,1)--(3,1)--(3,0)--(4,0)--(4,1)--(5,1)--(5,3)--(6,3)--(6,4)--(0,4)--cycle;
\draw[thinLine](0,1)--(1,1)(0,2)--(3,2)(0,3)--(5,3)(0,4)--(6,4)(0,5)--(6,5)(1,1)--(1,6)(2,1)--(2,6)(3,2)--(3,6)(4,2)--(4,6)(5,3)--(5,6)(6,5)--(6,6);
\draw[thickLine](0,0)--(1,0)--(1,1)--(3,1)--(3,2)--(5,2)--(5,3)--(6,3)--(6,5)--(7,5)--(7,6)--(0,6)--cycle;
\draw[entry](5,5)node{$4$}(5,4)node{$2$}(5,3)node{$1$}(0,3)node{$9$}(1,3)node{$7$}(2,3)node[circled]{$6$}(3,3)node{$4$}(4,3)node[circled]{$3$}(2,5)node{$9$}(2,4)node{$7$}(2,2)node{$4$}(2,1)node{$1$}(0,1)node{$4$}(1,1)node{$2$};
\end{scope}
\end{tikzpicture}
\caption{}
\label{add_cell}
\end{subfigure}
\caption{Matching the hook-lengths of a strict partition and its shift-symmetric partition.}
\label[fig]{proof_HL}
\end{figure}

\begin{prob} The fact that all quantities in \eqref{eq:fL} have natural combinatorial interpretations suggests that there might be a purely bijective proof of \cref{fL} that relies only on the manipulation of tableaux.
The authors are unaware of such a proof and it would be interesting to see one.
\end{prob}

The following two results provide an estimate for the probability that a fixed sub-diagram $\shyd{\lambda}$ of size $k$ of the shifted staircase contains precisely the entries $1,\dots,k$ in a shifted SYT of shape $\shyd{\Delta_n}$ chosen uniformly at random.


Let $m,n\in\N$, $\Box=(n^m)$, and $\lambda$ be a partition.
If $\yd{\lambda}\subseteq\yd{\Box}$ define a partition $\Box\setminus\lambda$ by setting
\begin{equation*}
\yd{(\Box\setminus\lambda)}
=\big\{(m-i+1,n-j+1):(i,j)\in\yd{\Box}\setminus\yd{\lambda}\big\}
\,.
\end{equation*}
Moreover, if $\lambda$ is strict and $\shyd{\lambda}\subseteq\shyd{\Delta_n}$, define a strict partition $\Delta_n\setminus\lambda$ by setting
\begin{equation*}
\shyd{(\Delta_n\setminus\lambda)}
=\big\{(n-j,n-i):(i,j)\in\shyd{\Delta_n}\setminus\shyd{\lambda}\big\}.
\end{equation*}

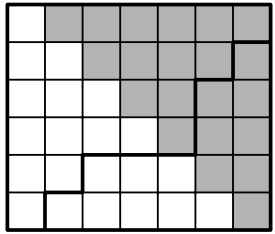
\begin{figure}[b]
\centering
\begin{tikzpicture}
\begin{scope}
\fill[fillGrey,xshift=1cm](0,5)--(1,5)--(1,4)--(2,4)--(2,3)--(3,3)--(3,2)--(4,2)--(4,1)--(5,1)--(5,0)--(6,0)--(6,6)--(0,6)--cycle;
\draw[thinLine](0,0)grid(7,6);
\draw[thickLine](0,0)rectangle(7,6)(1,0)--(1,1)--(2,1)--(2,2)--(5,2)--(5,4)--(6,4)--(6,5)--(7,5);
\end{scope}
\end{tikzpicture}
\caption{We have $\Box_7=(7^6)$, $\lambda=(6,4,2,1)$, $\Lambda=(7,6,5,5,2,1)$, and $\Box_7\setminus\Lambda=(6,5,2,2,1)$ is the shift-symmetric partition of $\Delta_7\setminus\lambda=(5,3)$.}
\label[fig]{shift_symm_rect}
\end{figure}

\begin{lem}\label[lem]{prob}
Let $n\in\N$, $\lambda$ be a strict partition of size $k$ with $\lambda_1<n$, and let $\P_n$ denote the uniform probability measure on $\T{n}$.
Then
\begin{equation}\label{eq:prob}
\P_n\big(\tab\in\T{n}:\tab(\shyd{\lambda})=[k]\big)
=\binom{N}{k}^{-1}\cdot\sqrt{\binom{2N}{2k}\frac{f_\Lambda\cdot f_{\Box_n\setminus \Lambda}}{f_{\Box_n}}}
\,,
\end{equation}
where $\Lambda$ and $\Box_n$ are the shift-symmetric partitions corresponding to $\lambda$ and $\Delta_n$ respectively.
\end{lem}

\begin{proof}
The left hand side of \eqref{eq:prob} is equal to
\begin{equation}\label{eq:prob1}
\P_n\big(\tab\in\T{n}:\tab(\shyd{\lambda})=[k]\big)
=\frac{\shf{\lambda}\cdot\shf{\Delta_n\setminus\lambda}}{\shf{\Delta_n}}
\,.
\end{equation}
It is not difficult to see that $\Box_n\setminus\Lambda$ is the shift-symmetric partition of $\Delta_n\setminus\lambda$.
See \cref{shift_symm_rect}. 
Thus \eqref{eq:prob1} can be computed by means of \cref{fL} above.
\end{proof}

The right hand side of \eqref{eq:prob} is essentially the same as \cite[Eq.~(7)]{PR07}.
This allows us to prove an analogue of \cite[Lem.~1]{PR07} for the shifted staircase.

\smallskip
Fix $n$ and let $\lambda$ be a partition with $\abs{\lambda}=k$.
Define a function $\gamma_{\lambda}:\R_{>0}\to\R_{\geq0}$ by
\begin{equation*}
\gamma_{\lambda}(x)=\frac{1}{n-1}\lambda_{\ceil{(n-1)x}}
\,,
\end{equation*}
where by convention $\lambda_i=0$ for $i>\ell(\lambda)$.
See \cref{russian}.
It is often more convenient to work with rotated coordinates
\begin{equation}\label{eq:uv}
u=\frac{x-y}{\sqrt{2}},\qquad
v=\frac{x+y}{\sqrt{2}}.
\end{equation}
Define $g_{\lambda}:\R\to\R_{\geq0}$ by
\begin{equation*}
g_{\lambda}(u)
=\sup\big\{v:v=\abs{u}\text{ or }x>0,y\leq\gamma_{\lambda}(x)\big\}
\,.
\end{equation*}
The function $g_{\lambda}$ is just a rotated version of the function $\gamma_{\lambda}$, however, it has the advantage of being $1$-Lipschitz while $\gamma_{\lambda}$ is only non-increasing.
Note that
\begin{equation*}
\int_{0}^{\infty}\gamma_{\lambda}(x)\matd x
=\int_{\R}g_{\lambda}(u)-\abs{u}\matd u
=\frac{k}{(n-1)^2}
\,.
\end{equation*}
Note that $g_{\lambda}$ describes the boundary of the scaled Young diagram $\yd{\lambda}$ in \emph{Russian convention}.
If the partition $\lambda$ is strict we define $G_{\lambda}:\R_{\leq0}\to\R_{\geq0}$ by
\begin{equation*}
G_{\lambda}(u)
=\sup\Big\{v:v=\abs{u}\text{ or }x>0,y\leq\gamma_{\lambda}(x)+\frac{\ceil{(n-1)x}-1}{(n-1)}\Big\}
\,.
\end{equation*}
The function $G_{\lambda}$ describes the boundary of the scaled shifted Young diagram $\shyd{\lambda}$.
See \cref{shifted_russian}.

\begin{figure}[t]
\centering
\begin{tikzpicture}
\begin{scope}
\draw[thinLine,black!30](0,1)--(3,1)(0,2)--(2,2)(0,3)--(2,3)(0,4)--(1,4)(0,5)--(1,5)(1,0)--(1,6)(2,0)--(2,4)(3,0)--(3,2)(4,0)--(4,1);
\draw[thinLine,->](0,0)--(0,7)node[anchor=south]{$y$};
\draw[thinLine,->](0,0)--(7,0)node[anchor=west]{$x$};
\draw[thickLine,blue](0,6)--(1,6)(1,4)--(2,4)(2,2)--(3,2)(3,1)--(4,1)(4,0)--(6,0);
\draw(4,0)node[anchor=north]{$\frac{4}{n-1}$};
\draw[blue](3,3)node{$\gamma_{\lambda}$};
\end{scope}
\begin{scope}[xshift=16cm,scale=.7071]
\draw[thinLine,black!30](-7,7)--(0,0)--(7,7)(-5,5)--(-4,6)(-4,4)--(-3,5)--(1,1)(-3,3)--(-1,5)(-2,2)--(0,4)--(2,2)(-1,1)--(2,4)--(3,3);
\draw[thinLine,->](0,0)--(0,9)node[anchor=south]{$v$};
\draw[thinLine,->](-8,0)--(8,0)node[anchor=west]{$u$};
\draw[thinLine,dashed]
(-6,0)node[anchor=north]{$-\frac{3\sqrt{2}}{n-1}$}--(-6,6)
(4,0)node[anchor=north]{$\frac{2\sqrt{2}}{n-1}$}--(4,4);
\draw[thickLine,blue](-7,7)--(-6,6)--(-5,7)--(-3,5)--(-2,6)--(0,4)--(1,5)--(2,4)--(3,5)--(4,4)--(7,7);
\draw[blue](2,6)node{$g_{\lambda}$};
\end{scope}
\end{tikzpicture}
\caption{The functions $\gamma_{\lambda}$ and $g_{\lambda}$ for the partition $(6,4,2,1)$.}
\label[fig]{russian}
\end{figure}
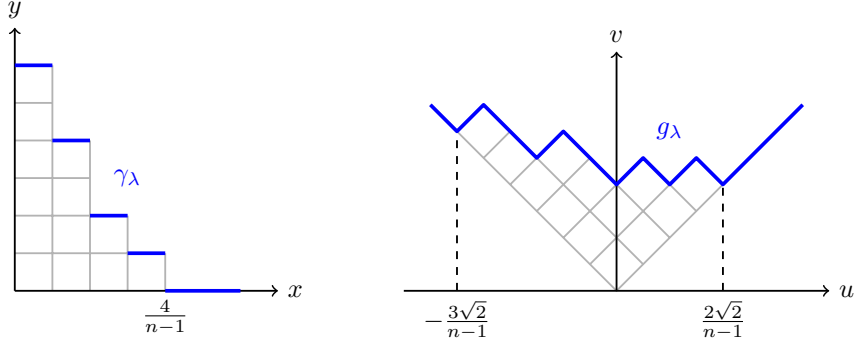

\begin{figure}[]
\centering
\begin{subfigure}{.4\textwidth}
\begin{tikzpicture}
\begin{scope}[scale=.7071]
\draw[thinLine,black!30](-8,8)--(0,0)--(1,1)(-5,5)--(-4,6)--(1,1)(-4,4)--(-2,6)--(1,3)--(-1,1)(-3,3)--(1,7)--(0,8)(-2,2)--(1,5)--(-1,7);
\draw[thinLine,->](0,0)--(0,9)node[anchor=south]{$v$};
\draw[thinLine,->](-9,0)--(3,0)node[anchor=west]{$u$};
\draw[thinLine,dashed]
(-6,0)node[anchor=north]{$-\frac{3\sqrt{2}}{n}$}--(-6,6)
(0,0)node[anchor=north]{$0$};
\draw[thickLine,blue](-8,8)--(-6,6)--(-5,7)--(-4,6)--(-3,7)--(-2,6)--(0,8);
\draw[blue](-4,8)node{$G_{\lambda}$};
\end{scope}
\end{tikzpicture}
\caption{The function $G_{\lambda}$ for the partition $(6,4,2,1)$.}
\label[fig]{shifted_russian}
\end{subfigure}
\begin{subfigure}{.5\textwidth}
\centering
\includegraphics[width=\textwidth]{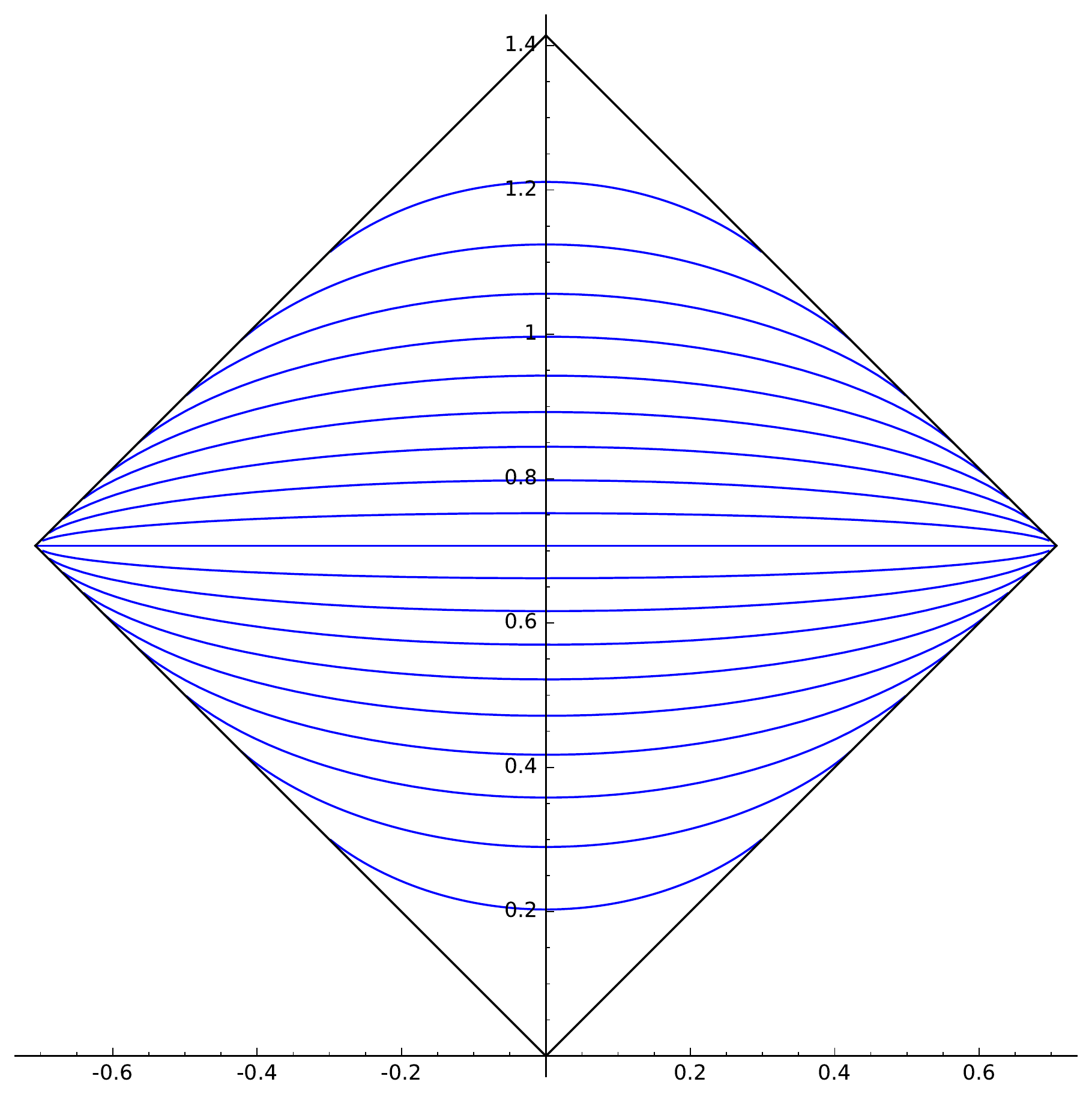}
\caption{The curves $v = g_{\alpha}(u)$ for $\alpha = 0.05, 0.1, \dots, 0.95$.}
\label[fig]{g_alpha}
\end{subfigure}
\caption{}
\end{figure}

\begin{lem}\label[lem]{variational_problem}
Let $\alpha\in(0,1)$, $k=k(n)$ be a sequence such that $k/N\to\alpha$ as $n\to\infty$, and let $\P_n$ denote the uniform probability measure on $\T{n}$.
Then, as $n\to\infty$,
\begin{equation*}
\P_n\big(\tab\in\T{n}:\tab(\shyd{\lambda})=[k]\big)=
\exp\Big(-\big(1+o(1)\big)\frac{n^2}{2}\big(I(\gamma_{\Lambda})+H(\alpha)+C\big)\Big)
\end{equation*}
uniformly over all strict partitions $\lambda$ of $k$ with $\lambda_1<n$, where 
\begin{equation*}
\begin{split}
C&=\frac{3}{2}-2\ln2,\\
H(\alpha)&=-\alpha\ln(\alpha)-(1-\alpha)\ln(1-\alpha),\\
I(\gamma)&=\int_0^1\int_0^1\ln\abs{\gamma(x)+\gamma^{-1}(y)-x-y}\matd y\matd x,\\
\gamma^{-1}(y)&=\inf\{x\in[0,1]:\gamma(x)\leq y\}
\,,
\end{split}
\end{equation*}
and $\Lambda$ denotes the shift-symmetric partition of $\lambda$.
\end{lem}

\begin{proof}
We use \cref{prob} and proceed exactly as in the proof of \cite[Lem.~1]{PR07}.
Note that
\begin{equation*}
-\ln\bigg(\binom{N}{k}^{-1}\cdot\sqrt{\binom{2N}{2k}}\bigg)
=-\frac{1}{4}\ln\big(\alpha(1-\alpha)\pi N\big)+\mathcal{O}(1)
\end{equation*}
as $n\to\infty$ by Stirling's approximation, thus these terms do not contribute to the analysis.
\end{proof}

Results of the type of \cref{variational_problem} lead to a so called large deviation principle.
Suppose that $\lambda$ is the strict partition of size $k$ with $\lambda_1<n$ such that its shift-symmetric partition $\Lambda$ minimises the integral $I(\gamma_{\Lambda})$.
If $\mu$ is a different strict partition of size $k$ with $\mu_1<n$ then by \cref{variational_problem} the probability that $\mu$ contains the numbers $1,\dots,k$ in a random shifted SYT $\tab\in\T{n}$ decays exponentially as $\mu$ deviates from $\lambda$.
This means that the shape formed by the entries $1,\dots,k$ in a random shifted SYT will be close to the minimising partition $\lambda$ with high probability.
One is therefore lead to the variational problem of identifying the function $\gamma$ within a certain search space depending on $\alpha$ that minimises the integral $I(\gamma)$.

A function
\begin{equation*}
g:[-\sqrt{2}/2,\sqrt{2}/2]\to[0,\sqrt{2}]
\end{equation*}
is called \emph{$\alpha$-admissible} if it is $1$-Lipschitz and satisfies
\begin{equation*}
\int_{-\sqrt{2}/2}^{\sqrt{2}/2}g(u)-\abs{u}\matd u=\alpha
\,.
\end{equation*}
As is explained in \cite[Sec.~2.2]{PR07} our problem is equivalent the following formulation:
For each $\alpha\in(0,1)$ find the unique $\alpha$-admissible function $g$ which is symmetric, that is, $g(-u)=g(u)$, and minimises the integral
\begin{equation}\label{eq:variational_K}
K(g)=-\frac{1}{2}
\int_{-\sqrt{2}/2}^{\sqrt{2}/2}
\int_{-\sqrt{2}/2}^{\sqrt{2}/2}
g'(s)g'(t)\ln\abs{s-t}\matd s\matd t
\,.
\end{equation}
The only difference between our situation and the situation in \cite{PR07} is the fact that our search space is smaller since we require that $\Lambda$ is the shift-symmetric partition of a strict partition.
In \cite[Sec.~2 and~3]{PR07} Pittel and Romik show that the variational problem \eqref{eq:variational_K} without the assumption $g(-u)=g(u)$ has the unique solution $\tilde{g}_{\alpha}$ given by
\begin{equation}\label[eq]{g_tilde}
\tilde{g}_{\alpha}(u)=
\begin{cases}
g_{\alpha}(u)&\quad\text{if }\abs{u}\leq\sqrt{2\alpha(1-\alpha)},\\
\abs{u}&\quad\text{if }\sqrt{2\alpha(1-\alpha)}\leq\abs{u}\leq\sqrt{2}/2,
\end{cases}
\end{equation}
where
\begin{equation*}
g_{\alpha}:[-\sqrt{2\alpha(1-\alpha)},\sqrt{2\alpha(1-\alpha)}]\to\R
\end{equation*}
is defined as
\begin{equation*}
g_{\alpha}(u)
=\frac{2u}{\pi}\tan^{-1}\Bigg(\frac{(1-2\alpha)u}{\sqrt{2\alpha(1-\alpha)-u^2}}\Bigg)
+\frac{\sqrt{2}}{\pi}\tan^{-1}\Bigg(\frac{\sqrt{2(2\alpha(1-\alpha)-u^2)}}{1-2\alpha}\Bigg)
\,,
\end{equation*}
if $0<\alpha<1/2$, and
\begin{equation*}
\tilde{g}_{\alpha}(u)=\sqrt{2}-\tilde{g}_{1-\alpha}(u)
\end{equation*}
for $1/2<\alpha<1$, and
\begin{equation*}
\tilde{g}_{1/2}(u)
=\frac{\sqrt{2}}{2}
\,.
\end{equation*}
The family of functions $g_{\alpha}$ is illustrated in \cref{g_alpha}. Since this solution already exhibits the additional symmetry $\tilde{g}_{\alpha}(-u)=\tilde{g}_{\alpha}(u)$, we may apply it to the shifted case as well.

Let
\begin{equation*}
\overline{L}:\big\{(u,v)\in\R^2:-\sqrt{2}/2\leq u\leq 0,\abs{u}\leq v\leq\sqrt{2}-\abs{u}\big\}\to\R_{\geq0}
\end{equation*}
be the surface defined by the level curves $v=g_{\alpha}(u)$ for $\alpha\in(0,1)$.
Let
\begin{equation*}
L:\big\{(x,y)\in\R^2:0\leq x\leq y\leq1\big\}\to\R_{\geq0}
\end{equation*}
be the rotated version of $\overline{L}$, that is, $L(x,y)=\overline{L}(u,v)$, where $(x,y)$ and $(u,v)$ are related as in \eqref{eq:uv}.

\smallskip
The following lemma collects analytic results on the integral $K(g)$ and the functions $\tilde{g}_{\alpha}$.

\begin{lem}\label{analysis}
\begin{enumerate}[(i)]
\item\label{Kcontinuous}
There exists a constant $c_K>0$ such that for all $1$-Lipschitz functions $g,h$ we have
\begin{equation*}
\abs{K(g-h)}\leq c_K\norm{g-h}
\,.
\end{equation*}
\item\label{ana1}
There exists a function $c:(2,3)\to\R_{>0}$ such that for all $r\in(2,3)$ and all $\alpha$-admissible functions $g$ we have
\begin{equation*}
K(g)+H(\alpha)-\ln2\geq c(r)\norm{g-\tilde{g}_{\alpha}}^r
\,.
\end{equation*}
\item\label{ana2}
Let $(x,y)\in(0,1)\times(0,1)$, let $(u,v)$ be given as in \eqref{eq:uv}, set $\alpha=\overline{L}(u,v)$, and set
\begin{equation*}
\sigma(x,y)=\min\big(xy,(1-x)(1-y)\big)
\,.
\end{equation*}
Then there exists constants $c_1>0$ and $c_2>0$ such that for all $\beta\in(0,1)$ and all $\delta<c_2\sigma(x,y)^2$ we have
\begin{equation*}
\abs{\tilde{g}_{\alpha}(u)-\tilde{g}_{\beta}(u)}<\delta c_1\sqrt{\sigma(x,y)}
\qquad\Rightarrow\qquad
\abs{\alpha-\beta}<\delta
\,.
\end{equation*}
\end{enumerate}
\end{lem}

\begin{proof}
Claim~\eqref{Kcontinuous} follows from the proof of \cite[Lem.~2]{PR07}.
Claim~\eqref{ana1} is an immediate consequence of \cite[Thm.~7 and Lem.~4]{PR07}.
Claim~\eqref{ana2} is precisely the statement of \cite[Lem.~5]{PR07}. 
\end{proof}

\smallskip
The last needed ingredient is a bound on the expected number of entries less than $k$ in the first row of a random shifted SYT.
We start with an auxiliary result. 
Given a partition $\lambda$, let $\lambda^+$ denote the partition obtained from $\lambda$ by adding a cell to the first row, that is, $\lambda^+=(\lambda_1+1,\lambda_2,\dots,\lambda_{\ell(\lambda)})$.

\begin{lem}\label[lem]{amusing}
Let $m,n\in\N$, $\Box=(n^m)$ be a rectangle and $\lambda$ a partition such that $\yd{(\lambda^+)}\subseteq\yd{\Box}$.
Then
\begin{equation*}
\frac{f_{\lambda}\cdot f_{\Box\setminus\lambda^+}}{f_{\lambda^+}\cdot f_{\Box\setminus\lambda}}
=\frac{(m+\lambda_1)(n-\lambda_1)}{(\abs{\lambda}+1)(mn-\abs{\lambda})}.
\end{equation*}
\end{lem}
\begin{proof}
This was proved by Pittel and Romik in the case where $\Box$ is an $n$ times $n$ square \cite[Eq.~(71)]{PR07}.
The proof of the generalisation to rectangles relies on the same idea.

\smallskip
First note that
\begin{equation}\label{eq:ffrac}
\frac{f_{\lambda}}{f_{\lambda^+}}
=\frac{\abs{\lambda}!}{(\abs{\lambda}+1)!}\prod_{u\in\yd{\lambda}}\frac{h_{\lambda^+}(u)}{h_{\lambda}(u)}
=\frac{1}{(\abs{\lambda}+1)}\prod_{j=1}^{\lambda_1}\frac{h_{\lambda^+}(1,j)}{h_{\lambda}(1,j)}
\,.
\end{equation}
Divide $[\lambda_1]$ into maximal sub-intervals $[i_r,j_r]$ such that $\lambda_j'=\lambda_{i_r}'$ for all $i_r\leq j\leq j_r$.
After even more cancellation \eqref{eq:ffrac} is equal to
\begin{equation*}
\frac{1}{(\abs{\lambda}+1)}\prod_{r}\frac{h_{\lambda^+}(1,i_r)}{h_{\lambda}(1,j_r)}
=\frac{1}{(\abs{\lambda}+1)}\prod_{r}\frac{\lambda_1-i_r+\lambda_{i_r}'}{\lambda_1-j_r+\lambda_{j_r}'}
\,,
\end{equation*}
which we rewrite as
\begin{equation}\label{eq:corners_f}
\frac{\chi}{(\abs{\lambda}+1)}
\prod_{(u_1,u_2)\in A}(\lambda_1-u_2+u_1)^{-1}
\prod_{(u_1,u_2)\in B}(\lambda_1-u_2+u_1)
\,,
\end{equation}
where $A$ denotes the set of cells $u\in\yd{\lambda}$ such that $\e u,\s u\notin\yd{\lambda}$, $B$ denotes the set of cells $u\in\yd{\Box}\setminus\yd{\lambda}$ such that $\n u,\w u\notin(\yd{\Box}\setminus\yd{\lambda})$, and
\begin{equation*}
\chi=
\begin{cases}
(m+\lambda_1)&\quad\text{if }\ell(\lambda)=m,\\
1&\quad\text{otherwise.}
\end{cases}
\end{equation*}

Divide the set $[m-1]$ into maximal sub-intervals $[i_r,j_r]$ such that $(\Box\setminus\lambda^+)_j=n-\lambda_{m+1-i_r}$ 
for all $i_r\leq j\leq j_r$.
Then
\begin{equation*}
\frac{f_{\Box\setminus\lambda^+}}{f_{\Box\setminus\lambda}}
=
\frac{n-\lambda_1}{(mn-\abs{\lambda})}\prod_{r}\frac{h_{\Box\setminus\lambda}(n-\lambda_1,i_r)}{h_{\Box\setminus\lambda^+}(n-\lambda_1,j_r)}
\,,
\end{equation*}
which can be written as
\begin{equation}\label{eq:corners_Boxf}
\frac{(n-\lambda_1)\cdot\overline{\chi}}{(mn-\abs{\lambda})}
\prod_{(u_1,u_2)\in A}(\lambda_1-u_2+u_1)
\prod_{(u_1,u_2)\in B}(\lambda_1-u_2+u_1)^{-1}
\,,
\end{equation}
where $A$ and $B$ are defined as above, and
\begin{equation*}
\overline{\chi}=
\begin{cases}
1&\quad\text{if }\ell(\lambda)=m,\\
(m+\lambda_1)&\quad\text{otherwise.}
\end{cases}
\end{equation*}
The claim now follows from the observation that almost all factors that appear in \eqref{eq:corners_f} and \eqref{eq:corners_Boxf} cancel.
\end{proof}

The following lemma provides us with an analogue of \cite[Eq.~(75)]{PR07} for the shifted case.
Given $n,k\in\N$ with $k<N$, let $I_{n,k}:\T{n}\to\{0,1\}$ denote the random variable that takes the value $1$ if the entry $k$ is contained in the first row, and $0$ otherwise.
Moreover, let $J_{n,k}=\sum_{i=1}^kI_{n,i}$ denote the number of entries at most $k$ in the first row of a shifted SYT.

\begin{lem}\label[lem]{expected}
Let $k,n\in\N$ with $k<N$, and let $\E_n$ denote the expected value with respect to the uniform probability measure $\P_n$ on $\T{n}$.
Then
\begin{equation*}
\E_n^2[I_{n,k}]
<\E_{n}\left[\frac{2N-(J_{n,k})^2}{k(N-k+1)}\right]
\,.
\end{equation*}
\end{lem}

\begin{proof}
Our proof is very similar to the first part of the proof of \cite[Lem.~10]{PR07}.

\smallskip
Let $\Y_{n,k}$ denote the set of all strict partitions $\lambda$ of size $k$ with $\lambda_1<n-1$.
Note that
\begin{equation*}
\begin{split}
\E_n[I_{n,k}]
&=\P_n\big(\tab\in\T{n}:\tab^{-1}(k)=(1,j)\text{ for some }j\in[n-1]\big)\\
&=\sum_{\lambda\in\Y_{n,k-1}}\frac{\shf{\lambda}\cdot\shf{\Delta_n\setminus\lambda^+}}{\shf{\Delta_n}}\\
&=\sum_{\lambda\in\Y_{n,k-1}}\frac{\shf{\lambda^+}\cdot\shf{\Delta_n\setminus\lambda^+}}{\shf{\Delta_n}}\cdot\frac{\shf{\lambda}}{\shf{\lambda^+}}.
\end{split}
\end{equation*}
Using the fact that $\mu\mapsto\shf{\mu}\cdot\shf{\Delta_n\setminus\mu}/\shf{\Delta_n}$ defines a probability measure on the set of strict partitions $\mu$ of size $k$ with $\mu_1<n$, and the convexity of the square function we obtain
\begin{equation}\label{eq:evil_sum}
\begin{split}
\E_n^2[I_{n,k}]
&\leq\sum_{\lambda\in\Y_{n,k-1}}\frac{\shf{\lambda^+}\cdot\shf{\Delta_n\setminus\lambda^+}}{\shf{\Delta_n}}\cdot\left(\frac{\shf{\lambda}}{\shf{\lambda^+}}\right)^2\\
&=\sum_{\lambda\in\Y_{n,k-1}}\frac{\shf{\lambda}\cdot\shf{\Delta_n\setminus\lambda}}{\shf{\Delta_n}}\cdot\frac{\shf{\lambda}\cdot\shf{\Delta_n\setminus\lambda^+}}{\shf{\lambda^+}\cdot\shf{\Delta_n\setminus\lambda}}.
\end{split}
\end{equation}
Let $L,M$ and $\Box_n$ denote the shift-symmetric partitions of $\lambda, \lambda^+$ and $\Delta_n$, respectively.
\cref{fL} yields
\begin{equation}\label{eq:four_fsh}
\begin{split}
\frac{\shf{\lambda}\cdot\shf{\Delta_n\setminus\lambda^+}}{\shf{\lambda^+}\cdot\shf{\Delta_n\setminus\lambda}}
&=4\cdot\sqrt{\frac{(k-\frac{1}{2})(N-k+\frac{1}{2})}{k(N-k+1)}}\cdot\sqrt{\frac{f_L\cdot f_{\Box_n\setminus M}}{f_M\cdot f_{\Box_n\setminus L}}}\\
&=4\cdot\sqrt{\frac{(k-\frac{1}{2})(N-k+\frac{1}{2})}{k(N-k+1)}}
\cdot\sqrt{\frac{f_L\cdot f_{\Box_n\setminus L^+}}{f_{L^+}\cdot f_{\Box_n\setminus L}}}
\cdot\sqrt{\frac{f_{(L^+)'}\cdot f_{(\Box_n\setminus M)'}}{f_{M'}\cdot f_{(\Box_n\setminus L^+)'}}}
\,.
\end{split}
\end{equation}
Using $M'=((L^+)')^+$ and \cref{amusing} twice we obtain
\begin{equation*}
\begin{split}
&\frac{f_L\cdot f_{\Box_n\setminus L^+}}{f_{L^+}\cdot f_{\Box_n\setminus L}}
\cdot\frac{f_{(L^+)'}\cdot f_{(\Box_n\setminus M)'}}{f_{M'}\cdot f_{(\Box_n\setminus L^+)'}}
=\frac{(n+\lambda_1)^2(n-\lambda_1-1)^2}{16(k-\frac{1}{2})(N-k+1)k(N-k+\frac{1}{2})}
\,.
\end{split}
\end{equation*}
Inserting this into \eqref{eq:four_fsh} we obtain
\begin{equation}\label{eq:no_more_roots}
\frac{(n+\lambda_1)(n-\lambda_1-1)}{k(N-k+1)}
\,.
\end{equation}
Combining \eqref{eq:evil_sum} and \eqref{eq:no_more_roots} yields the claim.
\end{proof}

We now prove the limit shape theorem for shifted SYT of staircase shape chosen uniformly at random.
Our result is an analogue of \cite[Thm.~1]{PR07}.
The obtained limit shape is the same as the limit shape for random SYT of square shape except that the domain is restricted from a square to a triangle.

In particular \eqref{eq:shape} provides point-wise convergence to the limit surface, while \eqref{eq:convergence_rate} specifies the rate of convergence if we assume a sufficient distance to the catheti.

\begin{thm}\label[thm]{limit_shape}
For $n\in\N$ let $\Delta_n$ denote the staircase partition of size $N=\binom{n}{2}$, $\T{n}$ the set of shifted SYT of shape $\shyd{\Delta_n}$, and $\P_n$ the uniform probability measure on $\T{n}$.
Then for all $\epsilon>0$
\begin{equation}\label{eq:shape}
\lim_{n\to\infty}\P_n\left(\tab\in\T{n}:
\max_{(i,j)\in\shyd{\Delta_n}}
\abs{\frac{\tab(i,j)}{N}
-L\Big(\frac{i}{n},\frac{j}{n}\Big)}
>\epsilon\right)=0
\,.
\end{equation}
Moreover for all $p\in(0,1/2)$ and all $q\in(0,p/2)$ such that $2p+q<1$
\begin{equation}\label{eq:convergence_rate}
\lim_{n\to\infty}\P_n\Bigg(\tab\in\T{n}:
\max_{\substack{(i,j)\in\shyd{\Delta_n}\\
\sigma(i/n,j/n)>n^{-q}}}
\abs{\frac{\tab(i,j)}{N}
-L\Big(\frac{i}{n},\frac{j}{n}\Big)}
>n^{-p}
\Bigg)=0
\,,
\end{equation}
where $\sigma(x,y)=\min\{xy,(1-x)(1-y)\}$.
\end{thm}

\begin{proof}
The first part \eqref{eq:shape} is proven in the same way as \cite[Thm.~1~(i)]{PR07} in \cite[Sec.~4]{PR07} with \cref{expected} taking the place of \cite[Eq.~(75)]{PR07}.
The proof of \eqref{eq:convergence_rate} is essentially the same as the proof of \cite[Thm.~1~(ii)]{PR07} given in \cite[Sec.~2.3]{PR07} using \cref{variational_problem} in place of \cite[Lem.~1]{PR07}.
Below we only demonstrate the details for \eqref{eq:convergence_rate}.

\smallskip
Let $p'(k)$ denote the number of strict partitions of size $k$.
Then
\begin{equation}\label{eq:pn}
p'(k)\sim
\frac{3^{3/4}}{12k^{3/4}}\exp\left(\pi\sqrt{k/3}\right)
\,,
\end{equation}
as $k\to\infty$.
Confer \cite[Fig.~I.9]{FlaSed}.

Given $k=\alpha N$ and a tableau $\tab\in\T{n}$ let $\lambda_{\tab,k}$ denote the partition with shifted Young diagram $\shyd{\lambda_{\tab,k}}=\tab^{-1}([k])$, and $\Lambda_{\tab,k}$ denote the shift-symmetric partition of $\lambda_{\tab,k}$.
Note that given the shift-symmetric partition $\Lambda$ of a strict partition $\lambda$, the function $g_{\Lambda}$ is $1$-Lipschitz but not $\alpha$-admissible, since $g_\Lambda(u)-|u|$ might be positive starting from
$-\frac{\sqrt 2}{2}\frac{n+1}{n}$.
However, we can always chose an $\alpha$-admissible function $\hat{g}_{\Lambda}$ such that
\begin{equation}\label{eq:to_admissible}
\norm{g_{\Lambda}-\hat{g}_{\Lambda}}
\leq\frac{\sqrt 2}{n}
\,.
\end{equation}
There exists a constant $C>0$ such that for all $r\in(2,3)$ and all $\epsilon_1>0$
\begin{equation}\label{eq:probk}
\begin{split}
&\P_n\big(\tab\in\T{n}:\norm{g_{\Lambda_{\tab,k}}-\tilde{g}_{\alpha}}>\epsilon_1\big)
\\&\quad
=\sum_{\substack{\shyd{\lambda}\subseteq\shyd{\Delta_n},\abs{\lambda}=k\\
\norm{g_{\Lambda}-\tilde{g}_{\alpha}}>\epsilon_1}}
\P_n\big(\tab\in\T{n}:\tab(\shyd{\lambda})=[k]\big)
\\&\quad\overset{\mathclap{\text{\cref{variational_problem}}}}{\leq}\qquad
p'(k)\max_{\substack{\shyd{\lambda}\subseteq\shyd{\Delta_n},\abs{\lambda}=k\\
\norm{g_{\Lambda}-\tilde{g}_{\alpha}}>\epsilon_1}}
\exp\Big(-(1+o(1))\frac{n^2}{2}\big(K(g_{\Lambda})+H(\alpha)-\ln2\big)\Big)
\\&\quad\overset{\mathclap{\substack{\text{\eqref{eq:to_admissible} and}\\\text{\cref{analysis}~\eqref{Kcontinuous}}}}}{=}\qquad
p'(k)\max_{\substack{\shyd{\lambda}\subseteq\shyd{\Delta_n},\abs{\lambda}=k\\
\norm{g_{\Lambda}-\tilde{g}_{\alpha}}>\epsilon_1}}
\exp\Big(-(1+o(1))\frac{n^2}{2}\big(K(\hat{g}_{\Lambda})+H(\alpha)-\ln2\big)\Big)
\\&\quad
\overset{\mathclap{\substack{\text{\eqref{eq:pn} and}\\\text{\cref{analysis}~\eqref{ana1}}}}}{\leq}\qquad
\exp\Big(Cn-\frac{c(r)}{2}n^2\epsilon_1^r\Big)
\,,
\end{split}
\end{equation}
as $n\to\infty$.

For $(i,j)\in\shyd{\Delta_n}$ set $\beta=L(i/n,j/n)$.
Given $\tab\in\T{n}$ set $\alpha_{\tab}=\tab(i,j)/N$.
For all $r\in(2,3)$ and all $\delta>0$ that satisfy
\begin{equation*}
\delta<c_2\sigma(i/n,j/n)^2
\,,
\end{equation*}
if $n$ is large enough, then
\begin{equation}\label{eq:probij}
\begin{split}
&\P_n\Bigg(\tab\in\T{n}:\abs{\frac{\tab(i,j)}{N}-L\Big(\frac{i}{n},\frac{j}{n}\Big)}>\delta\Bigg)
\\&\quad
=\P_n\Big(\tab\in\T{n}:\abs{\alpha_{\tab}-\beta}>\delta\Big)
\\&\quad
\overset{\mathclap{\text{\cref{analysis}~\eqref{ana2}}}}{\leq}\qquad
\P_n\Big(\tab\in\T{n}:\abs{\tilde{g}_{\alpha_\tab}(u)-\tilde{g}_{\beta}(u)}>\delta c_1\sqrt{\sigma(i/n,j/n)}\Big)
\\&\quad
=\P_n\Big(\tab\in\T{n}:\abs{G_{\lambda_{\tab,\tab(i,j)}}(u)-\tilde{g}_{\alpha_{\tab}}(u)}>\delta c_1\sqrt{\sigma(i/n,j/n)}\Big)
\\&\quad
\leq\P_n\Bigg(\tab\in\T{n}:\abs{g_{\Lambda_{\tab,\tab(i,j)}}(u)-\tilde{g}_{\alpha_{\tab}}(u)}>\delta c_1\sqrt{\sigma(i/n,j/n)}-\frac{\sqrt{2}}{n}\Bigg)
\\&\quad
\leq\P_n\Big(\tab\in\T{n}:\norm{g_{\Lambda_{\tab,\tab(i,j)}}-\tilde{g}_{\alpha_{\tab}}}>\frac{\delta c_1}{2}\sqrt{\sigma(i/n,j/n)}\Big)
\\&\quad
\overset{\mathclap{\text{\eqref{eq:probk}}}}{\leq}\,
\exp\Big(Cn-\frac{c(r)}{2}\left(\frac{c_1}{2}\right)^rn^2\left(\delta\sqrt{\sigma(i/n,j/n)}\right)^r\Big)
\,.
\end{split}
\end{equation}
Suppose $\delta=n^{-p}$ and $\sigma(i/n,j/n)>n^{-q}$ for some $p>0$ and $q>0$ such that $p>2q$ and
\begin{equation*}
p+\frac{q}{2}
<\frac{1}{2}-\epsilon_2
\,.
\end{equation*}
Then for $n$ large enough 
\begin{equation*}
\delta
=n^{-p}
<n^{-2q}
<c_2\sigma(i/n,j/n)^2
\end{equation*}
and
\begin{equation*}
\frac{\delta c_1}{2}\sqrt{\sigma(i/n,j/n)}
>\frac{c_1}{2}n^{-p-q/2}
>n^{-1/2+\epsilon_2}
\,.
\end{equation*}
For all $\epsilon_2>0$ there exists $\epsilon_3>0$ such that the choice
\begin{equation*}
\epsilon_1>n^{-1/2+\epsilon_2},\qquad
r=2+\epsilon_3
\end{equation*}
yields
\begin{equation*}
\lim_{n\to\infty}\exp\Big(Cn-\frac{c(r)}{2}n^2\epsilon_1^r\Big)
=0
\,.
\end{equation*}
Thus since the number of cells in $\shyd{\Delta_n}$ is only quadratic in $n$, we obtain \eqref{eq:convergence_rate} by taking the union bound in \eqref{eq:probij} over all possible cells.
\end{proof}

\section{Intermediate permutations}\label{S:permutations}
This section contains the derivation of the limit of intermediate permutation matrices in random 132-avoiding sorting networks, a parallel to \cite[Conj.~2]{AHRV07}.

The (Rothe) \emph{diagram} \[D(\sigma) = \big\{(i, j) \in \mathbb{N}^2 : 1 \leq i < i'\ \text{and}\ 1 \leq j < j'\ \text{for all}\ M(\sigma)_{i', j'} = 1 \big\}\] of a permutation $\sigma$ is the set of cells left unshaded when we shade all the cells weakly to the east and south of 1-entries in the permutation matrix $M(\sigma)$.
\begin{thm}[{\cite[Thm 3.1, Cor. 3.4]{LP18}}]\label[thm]{diag}
Let $w = w_1 \cdots w_N$ be a 132-avoiding sorting network. Then the shape of $Q_{w_1 \dots w_k}$ is $D(\sigma_k)$.
\end{thm}
The height-$\alpha$ level curve of random staircase SYT at time $\alpha$ with $0 \leq \alpha \leq 1$ is given by $\overline{L}(u, v) = \alpha \Leftrightarrow v = g_{\alpha}(u)$. Recall the definition of the extension $\tilde{g}_{\alpha}$ of $g_{\alpha}$ to the full interval $[-\sqrt{2}/2, \sqrt{2}/2]$ in \eqref{g_tilde}.
Hence, the diagram of the intermediate permutation matrix at time $\alpha$ scaled by $1/n$ is determined by the (rotated) level curve $\frac{x+y}{\sqrt{2}} = \tilde{g}_{\alpha}(\frac{x-y}{\sqrt{2}})$ which, since $Q_w$ is non-shifted, also has to be (un)shifted. This can be done by sending $y$ to $y + x$ since $y_{\text{shifted}} = y_{\text{non-shifted}} + x$. Thus we have the curve $\frac{2x + y}{\sqrt{2}} = \tilde{g}_{\alpha}(\frac{-y}{\sqrt{2}})$ and so the explicit formula $x = \frac{1}{\sqrt{2}}(\tilde{g}_{\alpha}(\frac{-y}{\sqrt{2}}) - \frac{y}{\sqrt{2}})$. See \cref{translation} for an illustration of the discussion above.
\begin{figure}[ht]
\centering
	\begin{subfigure}[b]{.285\textwidth}
		\includegraphics[width=1.0\textwidth]{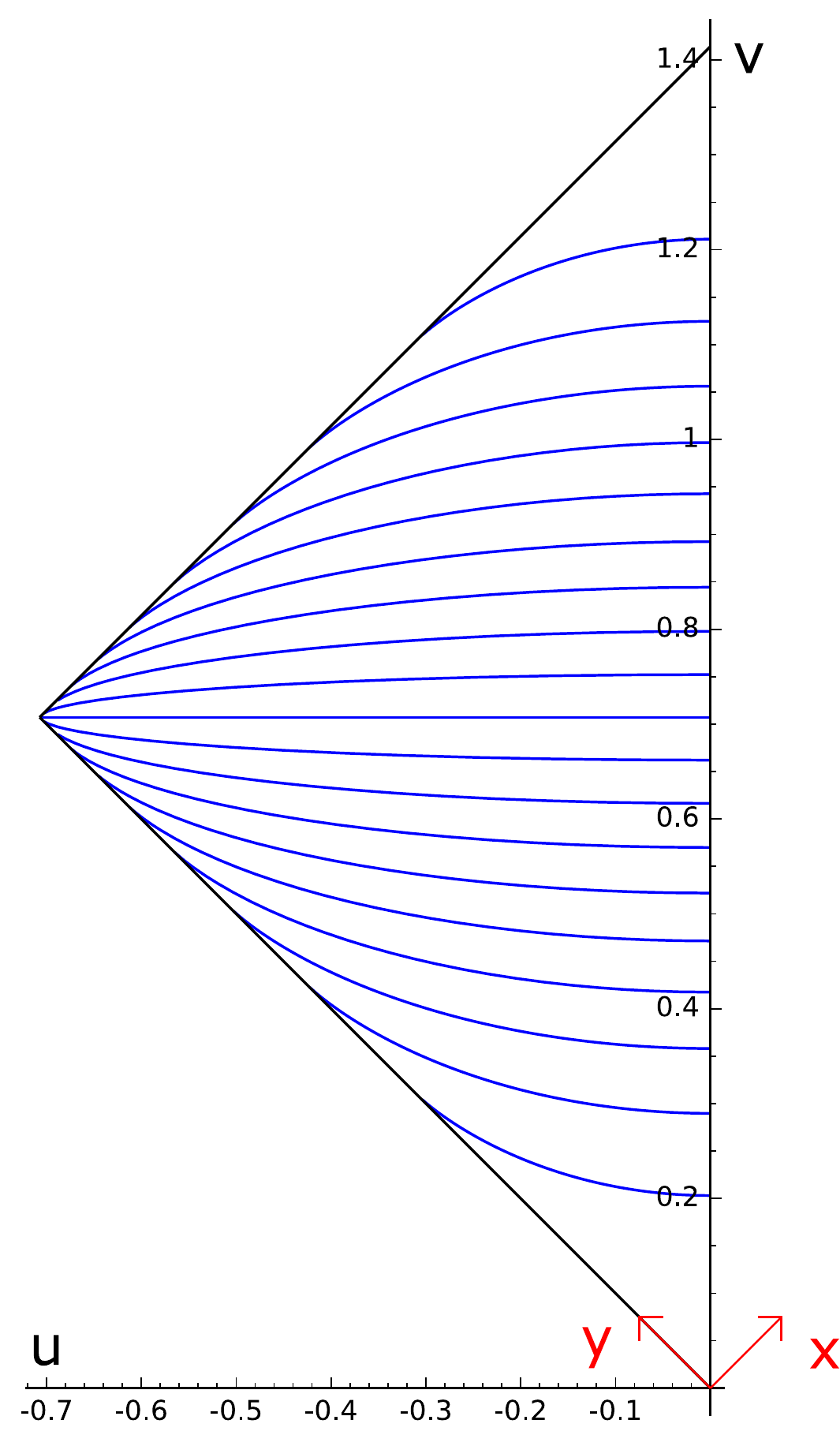}
		\caption{The curves $v = \tilde{g}_{\alpha}(u)$, $\alpha = 0.05, 0.1, \dots, 0.95$.}
	\end{subfigure}
	\quad
	\begin{subfigure}[b]{.315\textwidth}
		\includegraphics[width=1.0\textwidth]{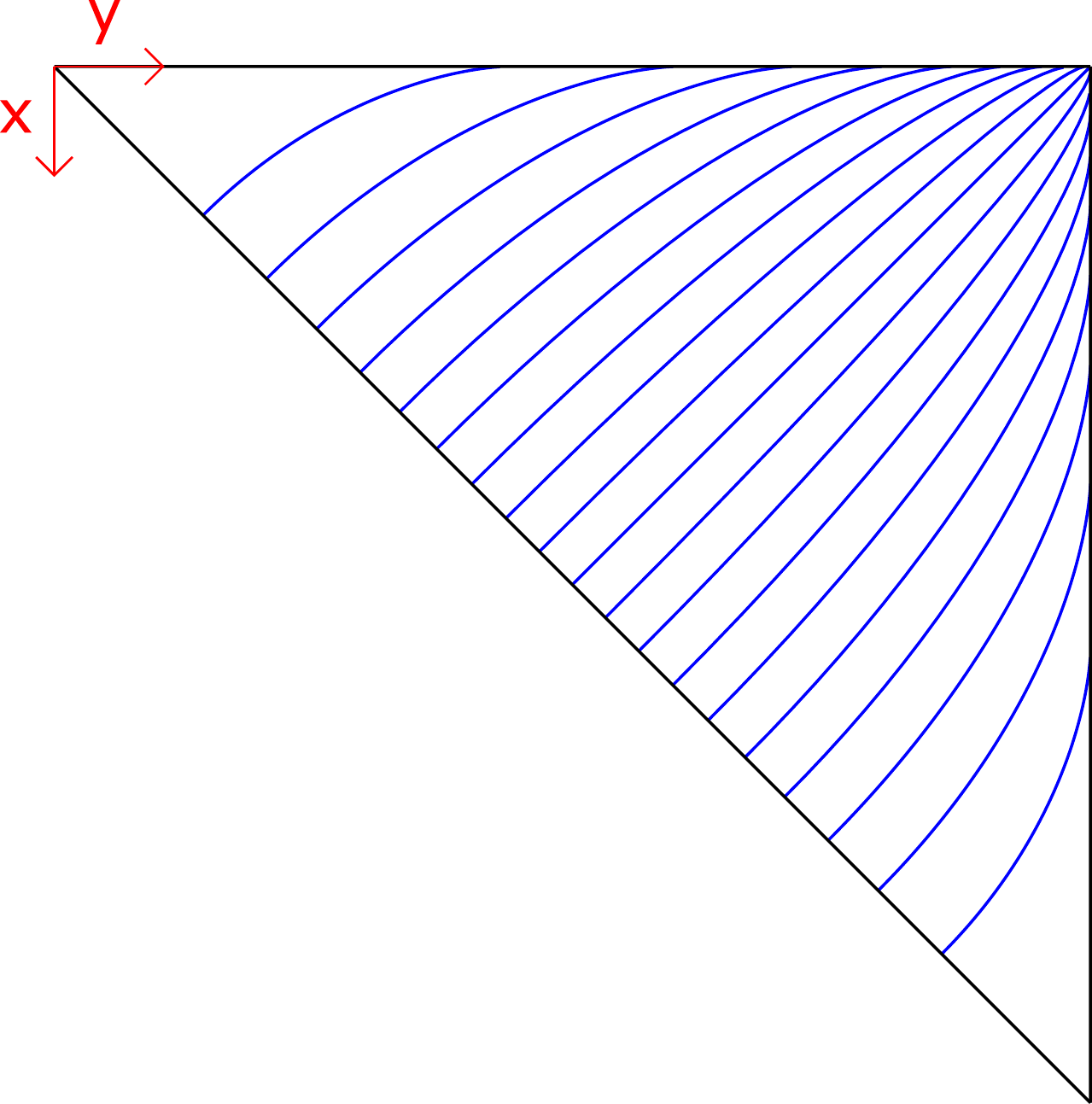}
		\caption{The previous picture after rotation.}
	\end{subfigure}
	\quad
	\begin{subfigure}[b]{.315\textwidth}
		\includegraphics[width=1.0\textwidth]{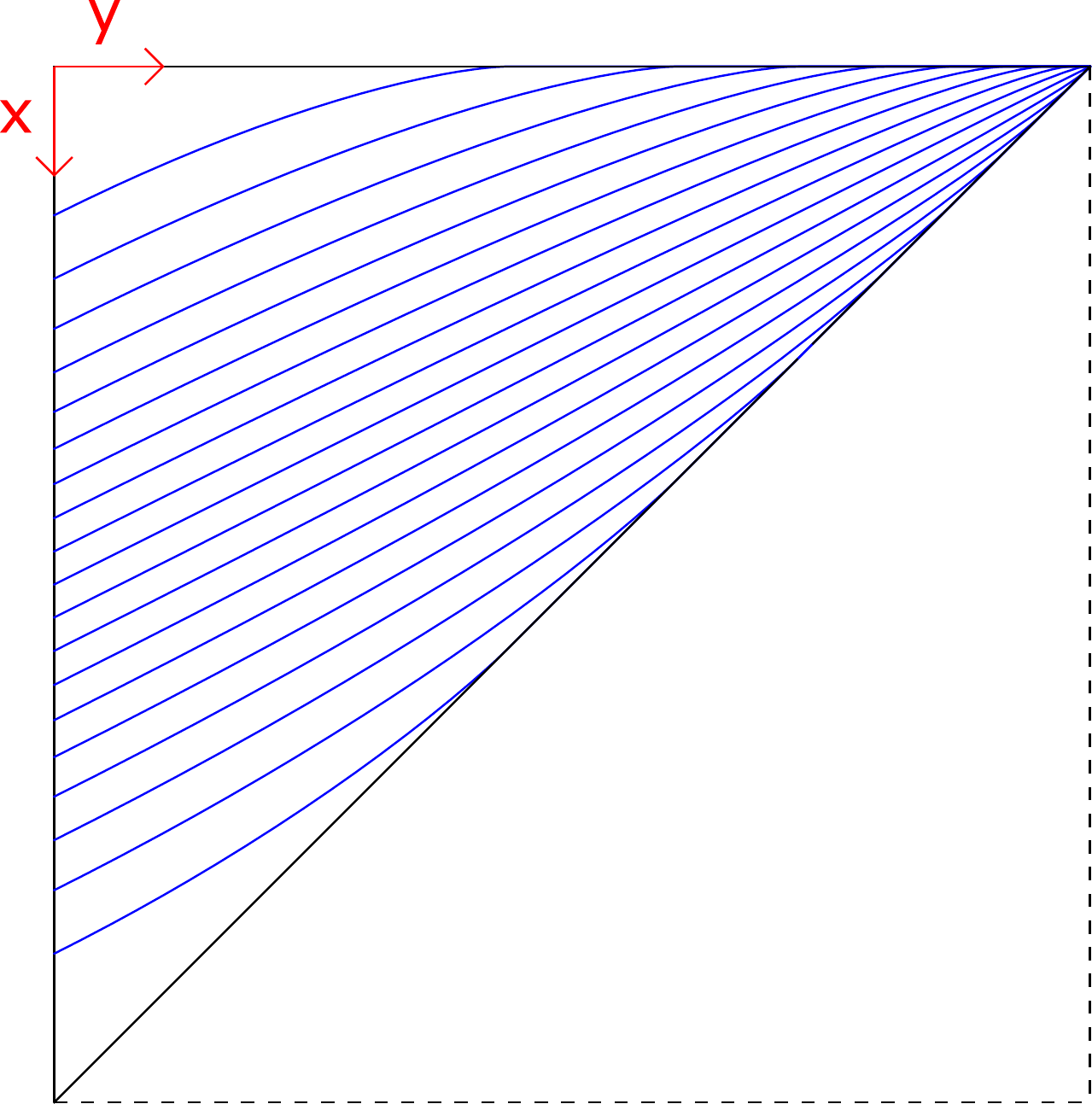}
		\caption{The previous picture after shifting to the left.}
		\label[fig]{transl_final}
	\end{subfigure}
	\caption{Translating the limit shape of shifted staircase SYT into the limit of the diagrams of the intermediate permutations. The blue curves in (c) are the limit curves of the diagrams of the intermediate permutations at times $\alpha = 0.05, 0.1, \dots, 0.95$.}
	\label[fig]{translation}
\end{figure}
\begin{thm}\label[thm]{intermediate_permutations}
Let $\sigma_0 = \text{id}$ and $\sigma_k = s_{w_1} \cdots s_{w_k}$ for $k \in [N]$, where $w = w_1 \dots w_N$ is a sorting network. Let $\P_n$ be the uniform probability measure on $\mathcal{R}^{132}_n$, the set of $n$-element 132-avoiding sorting networks. Finally, let
\[J_w(\alpha) = \big\{j \in [n] : \sigma_{\lfloor \alpha N \rfloor}(j) \leq \sigma_{\lfloor \alpha N \rfloor}(1)\big\}\] and $J_w^c(\alpha) = [n] \setminus J_n(\alpha)$. For all $0 \leq \alpha \leq 1$, $\epsilon > 0$, \[\P_n\biggl(w \in \mathcal{R}^{132}_n : \max_{j \in J_w(\alpha)}\biggl\vert\frac{\sigma_{\lfloor \alpha N \rfloor}(j)}{n} - \frac{1}{\sqrt{2}}\biggl(\tilde{g}_{\alpha}\biggl(\frac{-j}{n\sqrt{2}}\biggr) - \frac{j}{n\sqrt{2}}\biggr)\biggr\vert > \epsilon\biggr) \rightarrow 0,\] as $n \rightarrow \infty$. By symmetry, for all $0 \leq \alpha \leq 1$, $\epsilon > 0$, \[\P_n\biggl(w \in \mathcal{R}^{132}_n : \max_{j \in J_w^c(\alpha)}\biggl\vert\frac{\sigma_{\lfloor \alpha N \rfloor}(j)}{n} + \frac{1}{\sqrt{2}}\biggl(\tilde{g}_{1 - \alpha}\biggl(\frac{-j}{n\sqrt{2}}\biggr) - \frac{j}{n\sqrt{2}}\biggr) - 1\biggr\vert > \epsilon\biggr) \rightarrow 0,\] as $n \rightarrow \infty.$\end{thm}
\begin{proof}
\cref{diag} and \cref{limit_shape} together with the symmetry \cref{symmetry} imply the result by the discussion above.
\end{proof}
In particular, at $\alpha = \frac{1}{2}$ the diagram is bounded by the line $y = 1 - 2x$, which can be seen in \cref{matrix_sh}. Note that $(x, y) = 0$ is in the top-left corner.
\begin{figure}[ht]
	\begin{center}
		\includegraphics[width=.305\textwidth]{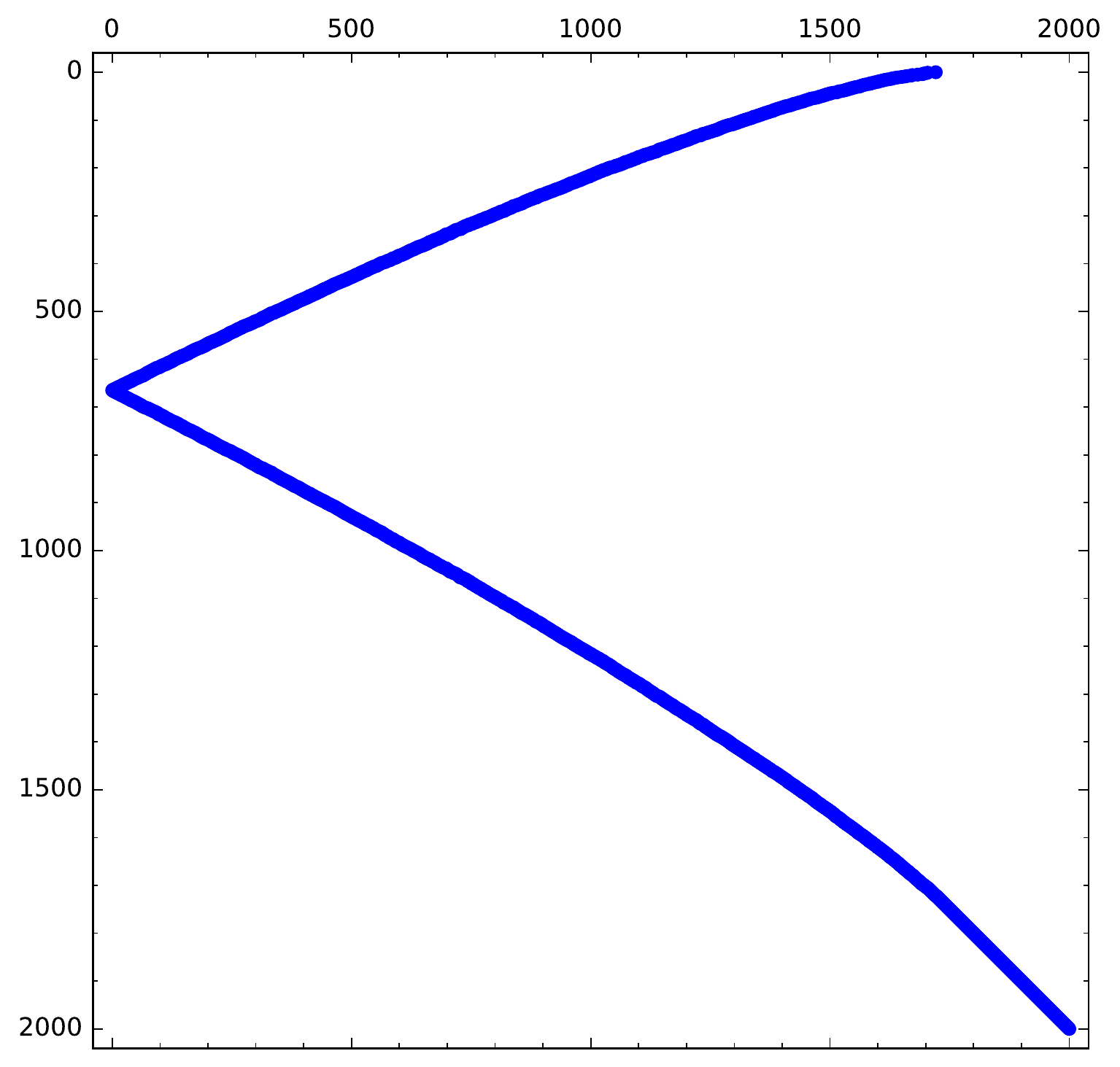}
		\quad
		\includegraphics[width=.305\textwidth]{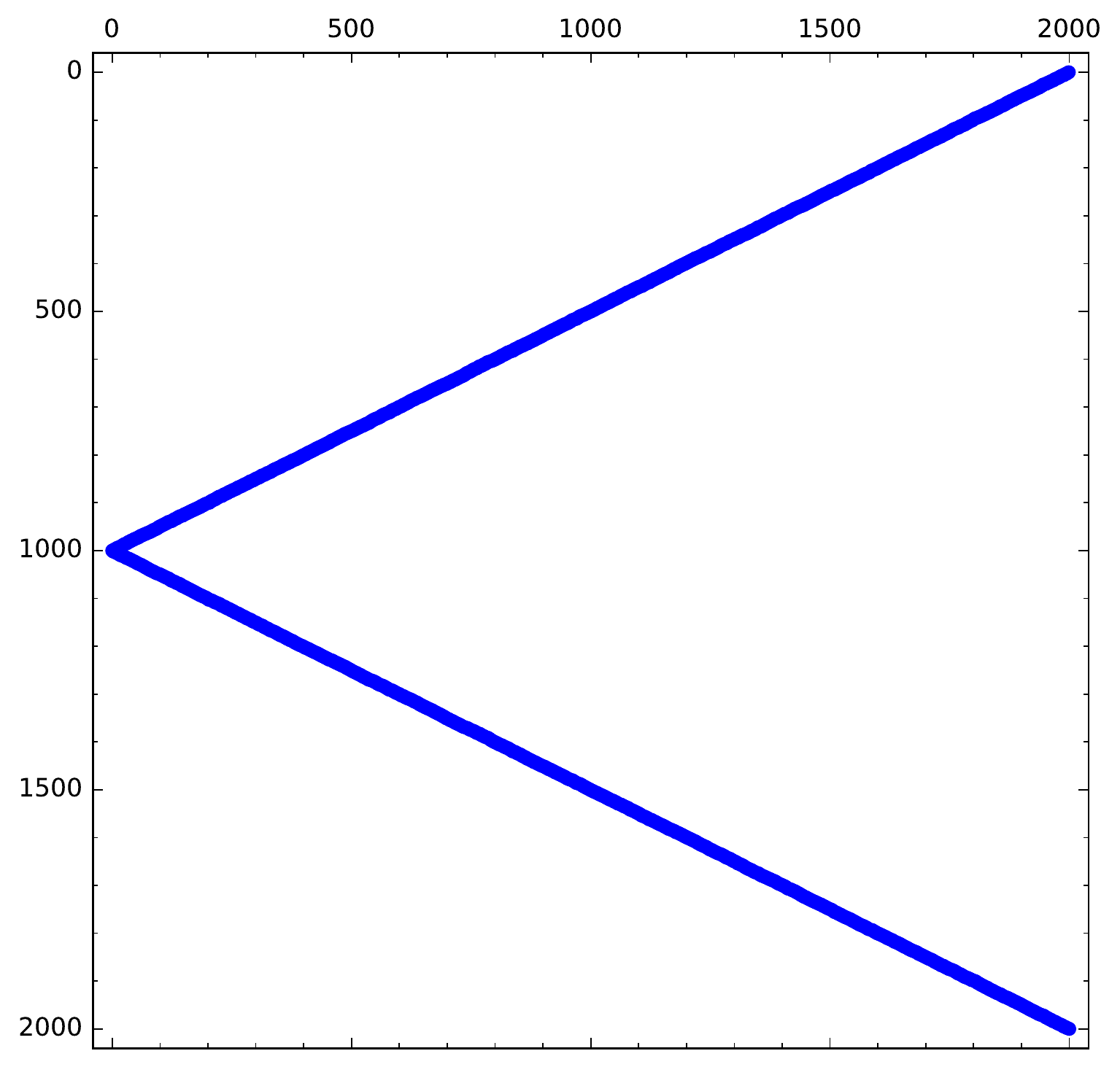}
		\quad
		\includegraphics[width=.305\textwidth]{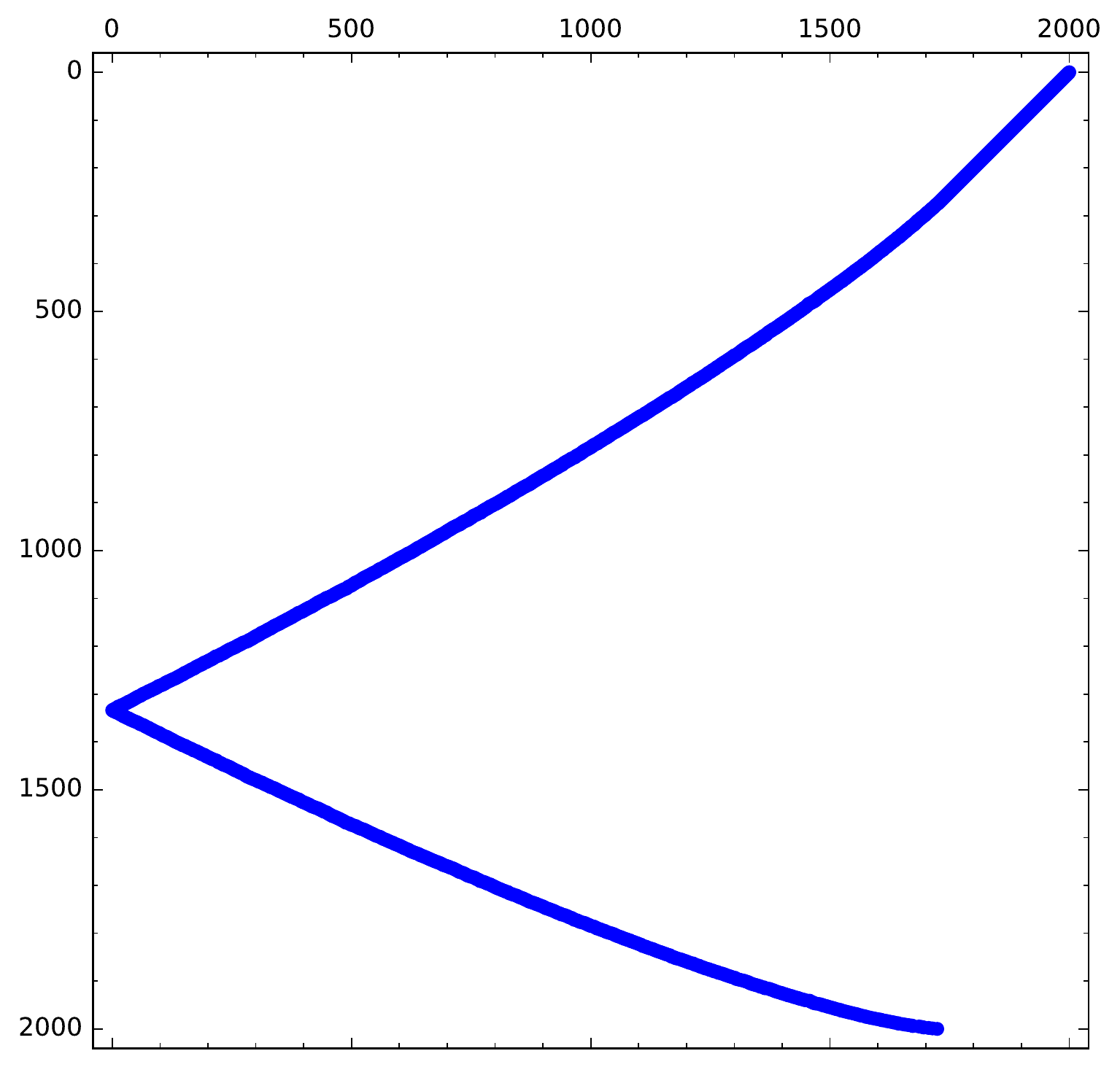}
	\end{center}
	\caption{The intermediate permutation matrices of a random 132-avoiding sorting network with 1000 elements at times $\alpha = \frac{1}{4}, \frac{1}{2}$ and $\frac{3}{4}$. Compare with \cref{transl_final} which contains the upper parts of the blue curves.}
	\label[fig]{matrix_sh}
\end{figure}

\section{Trajectories}\label{S:trajectories}
Next, inspired by the sine trajectories conjecture \cite[Conj.~1]{AHRV07} of Angel et al., we study trajectories in random 132-avoiding sorting networks.

The \emph{trajectory} of the element $i \in [n]$ in $w = w_1 \dots w_N$ is the function $k \mapsto \sigma_k^{-1}(i)$. See \cref{traj_ex}.
\begin{figure}[ht]
	\begin{center}
		\includegraphics[width=1.0\textwidth]{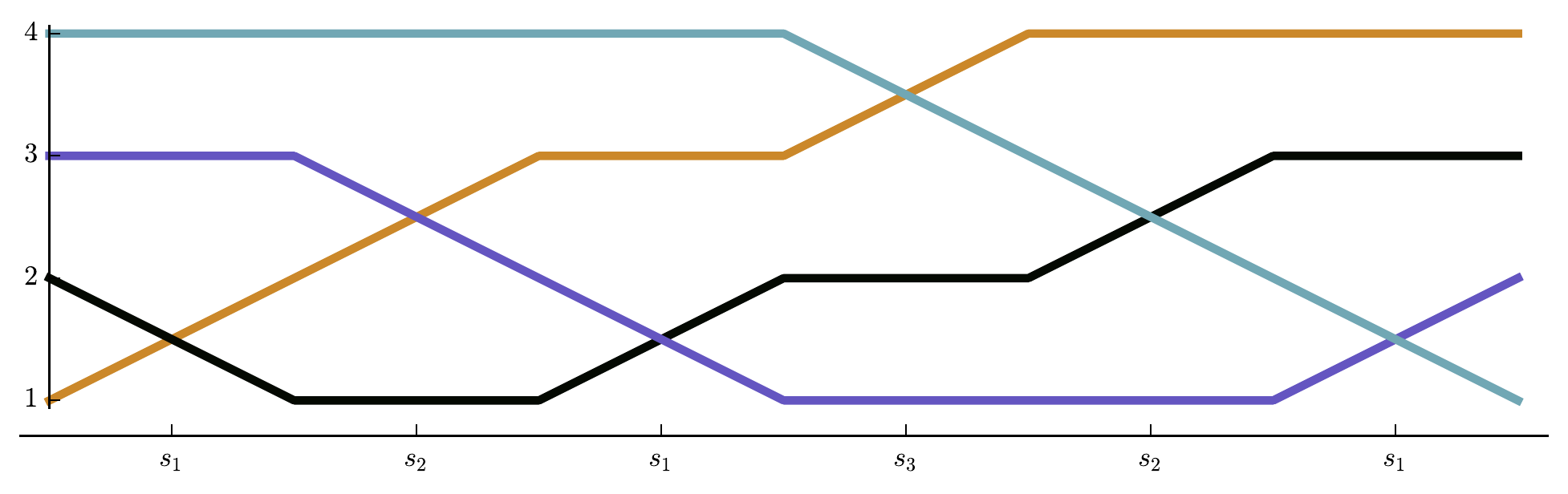}
	\end{center}
	\caption{Trajectories in the 132-avoiding sorting network 121321. The permutations $\sigma_k$ are $\sigma_0 = 1234$, $\sigma_1 = 2134$, $\sigma_2 = 2314$, $\sigma_3 = 3214$, $\sigma_4 = 3241$, $\sigma_5 = 3421$ and $\sigma_6 = 4 3 2 1$. Hence the trajectory of the element 3 is 3, 3, 2, 1, 1, 1, 2.}
	\label[fig]{traj_ex}
\end{figure}
The \emph{scaled trajectory} $f_i(\alpha) = f_{w,i}(\alpha)$ of $i$ in an $n$-element 132-avoiding sorting network $w$ is defined by \[f_i(\alpha) = \frac{\sigma^{-1}_{\alpha N}(i)}{n}\] for $\alpha N \in \Z$, and by linear interpolation for other $\alpha \in [0, 1]$. \cref{traj_sh} contains some examples.
\begin{figure}[b]
	\begin{center}
		\includegraphics[width=.65\textwidth]{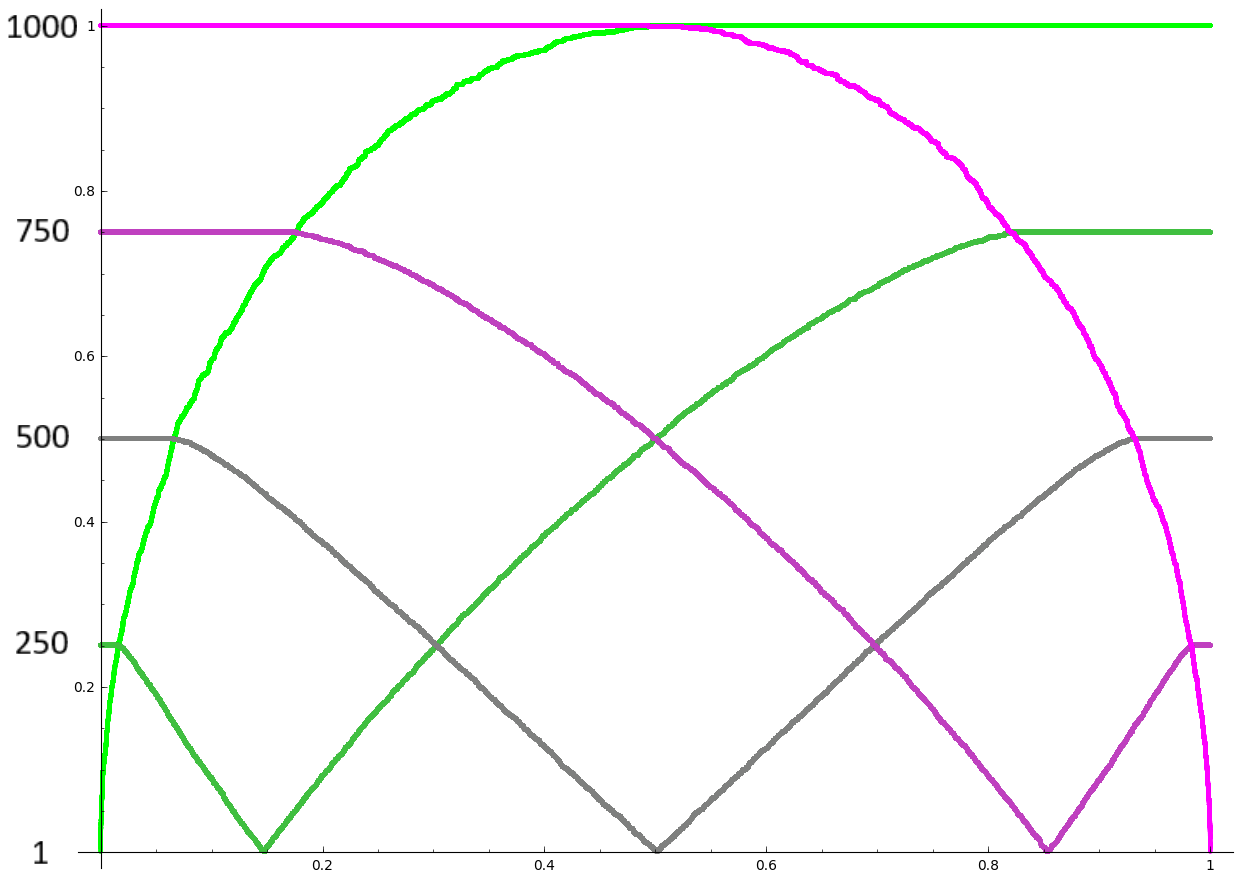}
	\end{center}
	\caption{The scaled trajectories of the elements 1, 250, 500, 750 and 1000 in a random 132-avoiding sorting network with 1000 elements.}
	\label[fig]{traj_sh}
\end{figure}

The heights at which different trajectories intersect are deterministic.
\begin{prop}
	In any 132-avoiding sorting network, $i$ and $j$ with $i < j$ are interchanged by $s_{j-i}$.
\end{prop}
\begin{proof}\label[prop]{intersection_heights}
We show that all elements between $i$ and $j$ have to pass $i$, and that all elements smaller than $i$ have to pass both $i$ and $j$ before $i$ and $j$ can be swapped. Note that elements $k$ such that $i < k < j$ cannot be swapped with $j$ before they have been swapped with $i$. Otherwise an intermediate permutation would have the pattern 132. Similarly, all $k'$ with $k' < i$ have to pass $i$ and hence also $j$ before $i$ and $j$ can be swapped.
\end{proof}

In general the element $k$ starts to move when 1 reaches position $k$, that is, when the first row of the tableau $Q_w$ has length $k-1$. \cref{traj_sh} shows how the trajectories are constant until they are intersected by the trajectory of 1. This time is given by the limit shape and is $(1-\sqrt{1-x^2})/2$, where $x=k/n$. Hence we have the result below.

\begin{prop}
	For $n\in\N$ let $\mathcal{R}_n^{132}$ denote the set of $n$-element 132-avoiding sorting networks and $\P_n$ be the uniform probability measure on $\mathcal{R}_n^{132}$.
	Then for all $\epsilon > 0$,
	\[\lim_{n\rightarrow \infty} \P_n\Big(w \in \mathcal{R}^{132}_n : \sup_{0 \leq \alpha \leq 1} \big|f_{w, 1}(\alpha) - \mathfrak{t}_1(\alpha)\big| > \epsilon\Big) = 0,\] where 
	\[\mathfrak{t}_1(\alpha) = \begin{cases} 
	2\sqrt{\alpha-\alpha^2} & \text{if}\ 0 \leq \alpha \leq \frac{1}{2},  \\
	1 & \text{if}\ \frac{1}{2} < \alpha \leq 1.
	\end{cases}\]
\end{prop}
\begin{proof} By the Edelman--Greene bijection, $f_{w, 1}(\alpha)$ is the length of the first row of $Q_w$ at time $\alpha$ scaled by $1/n$. By \cref{limit_shape}, for every $\epsilon > 0$ \begin{equation*}
\lim_{n\to\infty}\P_n\left(T\in \T{n}:\max_{j\in[n-1]}\abs{\frac{T(1,j)}{N}-\frac{1-\sqrt{1-(j/n)^2}}{2}}>\epsilon\right)=0.
\end{equation*} The function \[\frac{1-\sqrt{1-x^2}}{2}\] is continuous and strictly increasing. Hence its inverse is also continuous, so for every $\epsilon > 0$ there exists a $\delta > 0$ such that \[\abs{\frac{T(1, j)}{N}-\frac{1-\sqrt{1-(j/n)^2}}{2}} < \delta  \Rightarrow \abs{\frac{j}{n} - 2\sqrt{\frac{T(1, j)}{N}-\biggl(\frac{T(1, j)}{N}\biggr)^2}} < \epsilon.\] This proves the claim with $f_{w, 1}(\alpha) = j/n$ and $\alpha = T(1, j)/N$.
\end{proof}
By symmetry, the trajectory of $n$ is given by the transformation $\alpha \mapsto 1 - \alpha$.
In general, the trajectories are given by the limit shape. The observation below states that we can read the steps at which the element $m$ switches position and hence the position of $m$ at a given step by reading the labels of $Q_w$ first along the anti-diagonal $i + j - 1 = m-1$ and then along the column $i = m$. 
\begin{lem}\label{traj_general}
The trajectory of an element $m \in [n]$ in any $n$-element 132-avoiding sorting network $w$ is determined by $Q_w$ along the anti-diagonal $i + j - 1 = m-1$ and the row $i = m$. Namely, let \[D_m = \big\{(i, j) \in \Delta_n : i + j - 1 = m-1\ \text{or}\ i = m\big\}.\] Then \[\sigma_k^{-1}(m) = \begin{cases}m,\ &\text{if}\ k < Q_w(1, m),\\ \pr_2(Q_w^{-1}(\max\{k' \leq k:  Q_w^{-1}(k') \in D_m\})),\ &\text{otherwise},\end{cases}\] where $\pr_2(i, j) = j$.
\end{lem}
\begin{proof}
This follows from \cref{diag} by the observation that the element $m$ can only switch places with elements $m' < m$ until $m$ reaches the first column, and that $m' + 1$ cannot pass $m$ before $m'$ has, that is, the row $m'$ has reached length $m - m'$. This proves the part concerning the anti-diagonal $i+j-1 = m-1$. The remaining part follows by the symmetry \cref{symmetry}.
\end{proof}
\begin{thm}\label[thm]{scaled_trajectory}
Fix $m/n = \beta$. Let \[D_{\beta} = \big\{(x, y) \in \R^2 : 0 \leq x \leq y \leq 1, y = \beta\ \text{or}\ x = \beta\big\}.\] Define $\mathfrak{f}_{\beta}(\alpha) = y - x$, where $L^{-1}(\alpha) = (x, y) \in D_{\beta}$, and \[\mathfrak{t}_{\beta}(\alpha) = \begin{cases} 
	\beta & \text{if}\ 0 \leq \alpha \leq \frac{1-\sqrt{1-\beta^2}}{2}, \\
	\mathfrak{f}_{\beta}(\alpha) & \text{if}\ \frac{1-\sqrt{1-\beta^2}}{2} < \alpha < \frac{1+\sqrt{2\beta-\beta^2}}{2},\\
	1 - \beta & \text{if}\ \frac{1+\sqrt{2\beta-\beta^2}}{2} \leq \alpha \leq 1.
	\end{cases}\]
Finally, let $\mathcal{R}^{132}_n$ denote the set of $n$-element 132-avoiding sorting networks and let $\P_n$ be the uniform probability measure on $\mathcal{R}^{132}_n$.
	Then for all $\epsilon > 0$,
	\[\lim_{n\rightarrow \infty} \P_n\Big(w \in \mathcal{R}^{132}_n : \sup_{0 \leq \alpha \leq 1} \big|f_{w, \lfloor \beta n \rfloor}(\alpha) - \mathfrak{t}_{\beta}(\alpha)\big| > \epsilon\Big) = 0.\]
\end{thm}
\begin{proof}
This is simply a shifted and scaled version of \cref{traj_general} together with \cref{limit_shape}. Note that $L(x, y)$ restricted to $D_{\beta}$ is continuous by the continuity of $L(x, y)$ and also strictly increasing since $D_{\beta}$ intersects each level curve exactly once. 
\end{proof}
Informally, we can trace the trajectory of $\lfloor \beta n \rfloor$ by following the limit shape along $y = \beta$ until $x = y$, and then along $x = \beta$. If the height $\alpha = L(x', y')$ is given by some point $(x', y')$ along this curve, then the trajectory of $\lfloor \beta n\rfloor$ is at height $y'-x'$ at time $\alpha$. \cref{traj_shape} illustrates this. This combined with the implicit definition of $L(x, y)$ means that it is difficult to compute the trajectories of arbitrary elements explicitly.
\begin{figure}[t]
	\begin{center}
\includegraphics[width=.47\textwidth]{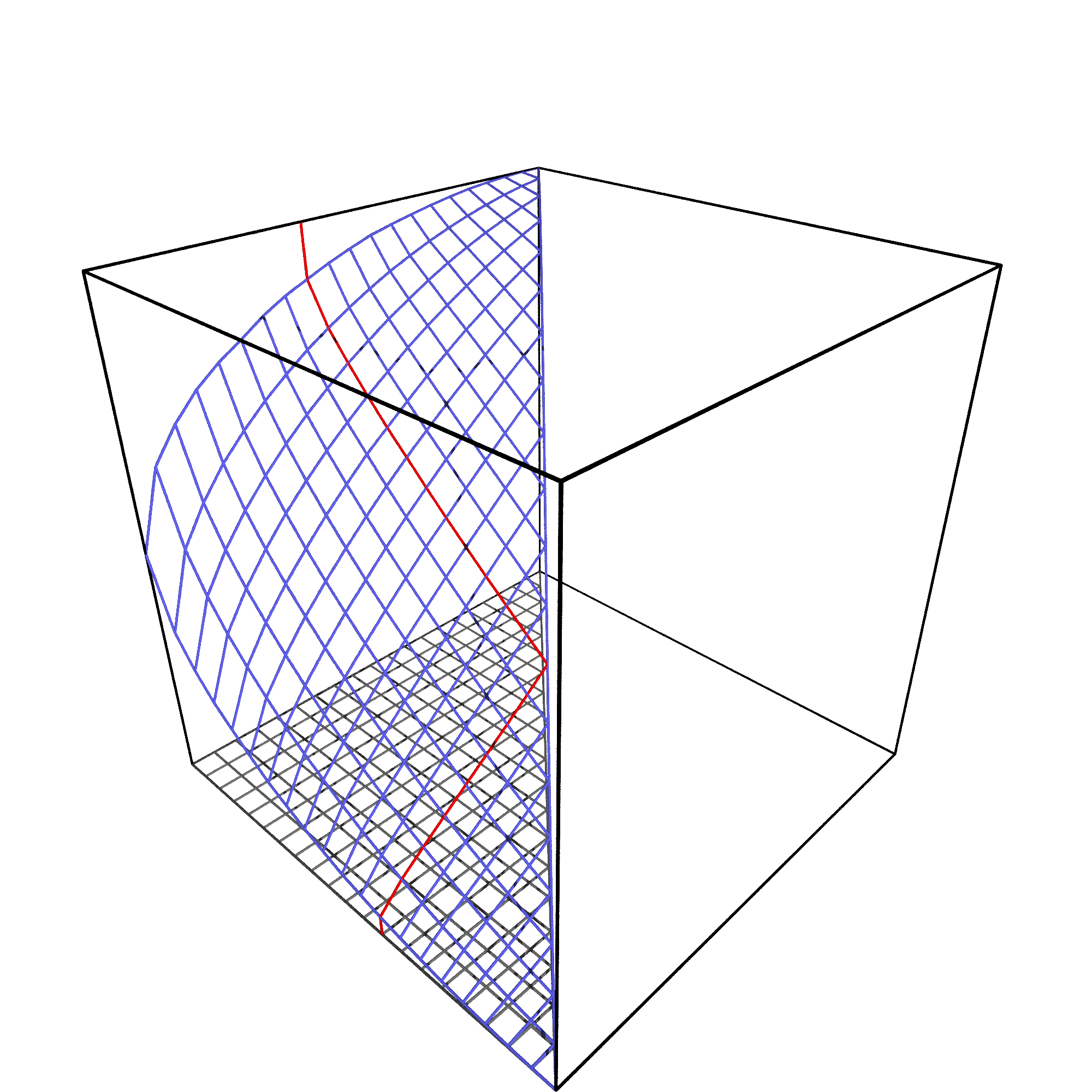}
\quad
\includegraphics[width=.4\textwidth]{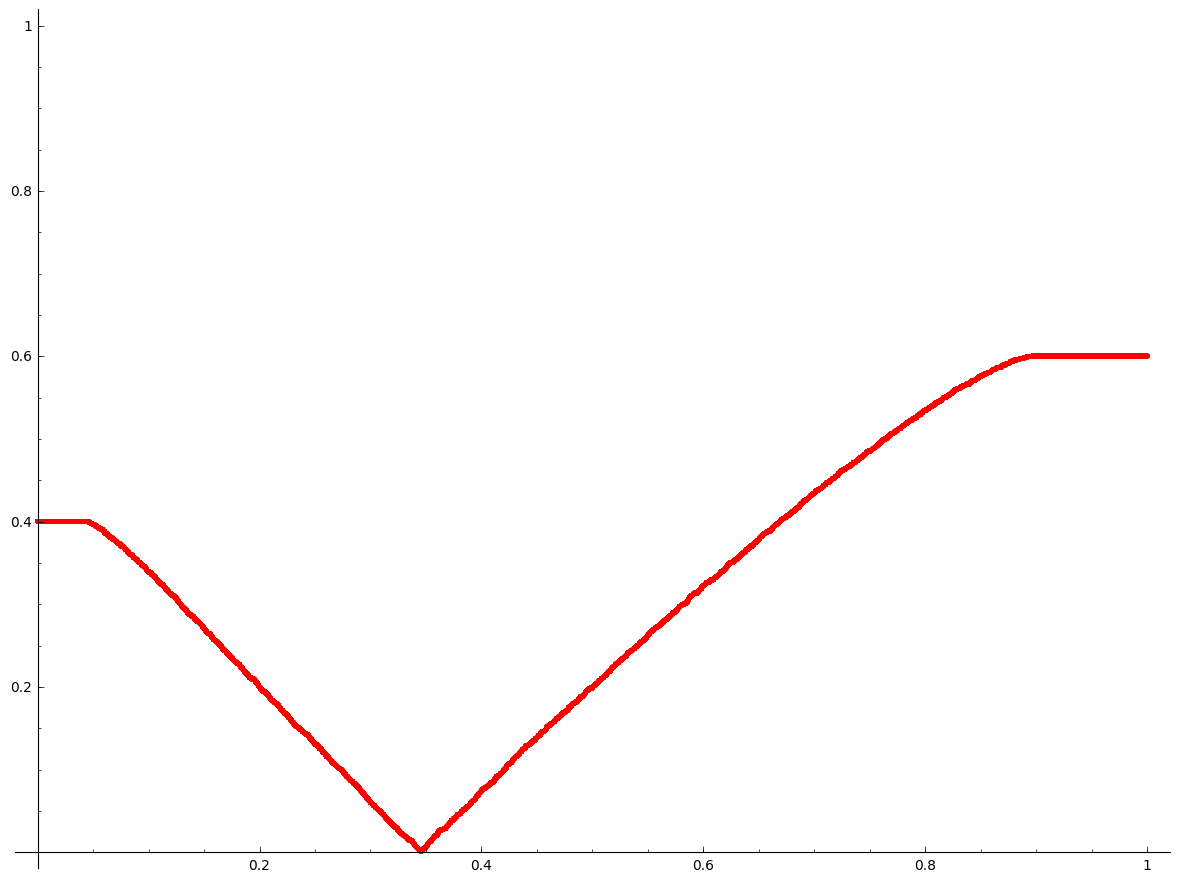}
\end{center}
	\caption{A curve (in red) determining the trajectory shown on the right from the limit shape.}
	\label[fig]{traj_shape}
\end{figure}

\section{Adjacencies}\label{S:adj}

Motivated by the great circle conjecture~\cite[Conj.~3]{AHRV07} and trying to understand the geometry of random 132-avoiding sorting networks on the permutahedron, we next study adjacencies. 

Let $w$ be a reduced word of the longest element in $\S_n$. An index $k \in [N-1]$ is called an \emph{adjacency} of $w$ if $\abs{w_{k+1}-w_k}=1$. In the permutahedron, an adjacency corresponds to a pair of adjacent edges of length $\sqrt{2}$ with direction vectors of the form $(0, \dots, 0, 1, 0, \dots, 0, -1, 0, \dots, 0)$ where the indices of either the 1s or the $-1$s coincide. Hence their scalar product is 1 and the edges constitute an angle of $\frac{\pi}{3}$. In the case of $\abs{w_{k+1}-w_k} > 1$ the edges corresponding to $w_k$ and $w_{k+1}$ are orthogonal. Adjacencies in a 132-avoiding sorting network $w$ correspond directly to adjacencies in the SYT $Q_w^{\rightarrow}$ as follows.

Let $\yd{\lambda}$ be a Young diagram.
A pair $(\tab,u)$ of a cell $u\in\yd{\lambda}$ and a standard tableaux $\tab$ of shape $\lambda$ is called a \emph{horizontal adjacency} if $\tab(\e u)=\tab(u)+1$.
The pair $(\tab,u)$ is called a \emph{vertical adjacency} if $\tab(\s u)=\tab(u)+1$.
An adjacency $(\tab,u)$ is said to lie in column $j$ of $\yd{\lambda}$ if $u=(i,j)$. Likewise, in such case $(\tab, u)$ is said to lie in row $i$ of $\yd{\lambda}$. The definitions are the same for shifted diagrams $\lambda^{\sh}$. For example, consider the tableau $\tab$ in \cref{tau}.
%
%
Then $(\tab, (1, 1))$, $(\tab, (1, 2))$ and $(\tab, (2, 2))$ are horizontal adjacencies whereas $(\tab, (2, 3))$ is a vertical adjacency. The adjacencies lie in columns 1, 2, 2, and 3 (and rows 1, 1, 2 and 2), respectively.

\begin{figure}[t]
\centering
\begin{tikzpicture}
\begin{scope}
\draw[thinLine](2,1)--(3,1)(1,2)--(3,2)(1,2)--(1,3)(2,1)--(2,3);
\draw[thickLine](0,2)--(1,2)--(1,1)--(2,1)--(2,0)--(3,0)--(3,3)--(0,3)--cycle;
\draw[entry](0,2)node{$1$}(1,2)node{$2$}(2,2)node{$3$}(1,1)node{$4$}(2,1)node{$5$}(2,0)node{$6$};
\draw(3.2,1.5)node[anchor=west]{$\tab=\tau_4(\tab)$};
\end{scope}
\begin{scope}[xshift=10cm]
\fill[fillGrey](1,1)rectangle(2,2)(2,2)rectangle(3,3);
\draw[thinLine](2,1)--(3,1)(1,2)--(3,2)(1,2)--(1,3)(2,1)--(2,3);
\draw[thickLine](0,2)--(1,2)--(1,1)--(2,1)--(2,0)--(3,0)--(3,3)--(0,3)--cycle;
\draw[entry](0,2)node{$1$}(1,2)node{$2$}(2,2)node{$4$}(1,1)node{$3$}(2,1)node{$5$}(2,0)node{$6$};
\draw(3.2,1.5)node[anchor=west]{$\tau_3(\tab)$};
\end{scope}
\end{tikzpicture}
\caption{The three shifted SYT $\tab,\tau_3(\tab)$ and $\tau_4(\tab)$.}
\label[fig]{tau}
\end{figure}
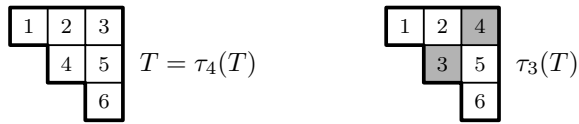

\begin{prop}
Let $w$ be a $132$-avoiding sorting network.
Then $(w_k,w_{k+1})=(j,j+1)$ if and only if $(Q_w^{\rightarrow},(Q_w^{\rightarrow})^{-1}(k))$ is a horizontal adjacency.
Similarly $(w_k,w_{k+1})=(j+1,j)$ if and only if $(Q_w^{\rightarrow},(Q_w^{\rightarrow})^{-1}(k))$ is a vertical adjacency.
\end{prop}
\begin{proof}
This follows from the fact that $w_i$ is inserted in column $w_i$ of $Q_w$ in the Edelman--Greene bijection, and that $Q_w$ can be shifted. 
\end{proof}

Our next goal is to enumerate adjacencies in Young tableaux.
To this end define \[\tau_{k}(\tab) = \begin{cases}
\tab &\quad \textrm{if}\ (\tab, \tab^{-1}(k))\ \textrm{is an adjacency}, \\
s_k \circ \tab &\quad \textrm{otherwise}.
\end{cases}\] In other words, $\tau_k(\tab)$ exchanges the positions of $k$ and $k+1$ if possible, that is, unless they are adjacent in $\tab$.
See \cref{tau}.

We consider the bijection $\partial_k = \tau_{|\lambda|-1} \circ \dots \circ \tau_{k+1} \circ \tau_{k}$ from the set of (possibly shifted) SYT of shape $\lambda$ to itself, where $|\lambda|$ is the largest entry in an SYT of shape $\lambda$. Applying $\partial_k$ to $\tab$ can be described by the following procedure called \emph{partial promotion}: \begin{enumerate} \item Form the sequence of cells $u_0, \dots, u_m$, called the \emph{promotion path}, such that \begin{itemize}
	\item $u_0 = \tab^{-1}(k)$,
	\item $u_{l+1} = \mathrm{arg} \min \{\tab(\e u_l), \tab(\s u_l)\}$. If at some $l = m$ both are undefined, we stop.
\end{itemize} \item Remove the label of $u_0$ and slide the entries $\tab(u_0) \leftarrow \tab(u_1) \leftarrow \dots \leftarrow \tab(u_{m})$. \item Subtract 1 from each label at least $k+1$ and insert $|\lambda|$ into $u_{m}$.
\end{enumerate} The inverse can be described in a similar way: 
\begin{enumerate} \item Form the sequence of cells $u_m, \dots, u_0$, called the \emph{inverse promotion path}, such that \begin{itemize}
	\item $u_m = \tab^{-1}(|\lambda|)$,
	\item $u_{l+1} = \mathrm{arg} \min \{\tab(\w u_l), \tab(\n u_l)\}$,
	\item and $u_0$ is the last possible cell in the sequence such that $\tab(u_0) \geq k$.
\end{itemize} \item Remove the label of $u_m$ and slide the entries $\tab(u_{0}) \rightarrow \dots \rightarrow \tab(u_{m})$. \item Add 1 to each label at least $k$. Insert $k$ into $u_0$.
\end{enumerate}

\begin{figure}[t]
\centering
\begin{tikzpicture}
\begin{scope}
\draw[xshift=46mm,yshift=15mm,thickLine,->](0,0)--node[anchor=south]{$\partial_2$}(1,0);
\begin{scope}
\fill[fillGrey](1,4)--(1,2)--(3,2)--(3,0)--(4,0)--(4,3)--(2,3)--(2,4)--cycle;
\draw[thinLine](3,1)--(4,1)(2,2)--(4,2)(1,3)--(4,3)(1,3)--(1,4)(2,2)--(2,4)(3,1)--(3,4);
\draw[thickLine](0,3)--(1,3)--(1,2)--(2,2)--(2,1)--(3,1)--(3,0)--(4,0)--(4,4)--(0,4)--cycle;
\draw[entry](0,3)node{$1$}(1,3)node{$2$}(2,3)node{$4$}(3,3)node{$5$}(1,2)node{$3$}(2,2)node{$6$}(3,2)node{$7$}(2,1)node{$8$}(3,1)node{$9$}(3,0)node{$10$};
\end{scope}
\begin{scope}[xshift=6cm]
\draw[thinLine](3,1)--(4,1)(2,2)--(4,2)(1,3)--(4,3)(1,3)--(1,4)(2,2)--(2,4)(3,1)--(3,4);
\draw[thickLine](0,3)--(1,3)--(1,2)--(2,2)--(2,1)--(3,1)--(3,0)--(4,0)--(4,4)--(0,4)--cycle;
\draw[entry](0,3)node{$1$}(1,3)node{$2$}(2,3)node{$3$}(3,3)node{$4$}(1,2)node{$5$}(2,2)node{$6$}(3,2)node{$8$}(2,1)node{$7$}(3,1)node{$9$}(3,0)node{$10$};
\end{scope}
\end{scope}
\begin{scope}[xshift=15cm]
\draw[xshift=46mm,yshift=15mm,thickLine,->](0,0)--node[anchor=south]{$\partial_2^{-1}$}(1,0);
\begin{scope}
\fill[fillGrey](1,4)--(1,2)--(3,2)--(3,0)--(4,0)--(4,3)--(2,3)--(2,4)--cycle;
\draw[thinLine](3,1)--(4,1)(2,2)--(4,2)(1,3)--(4,3)(1,3)--(1,4)(2,2)--(2,4)(3,1)--(3,4);
\draw[thickLine](0,3)--(1,3)--(1,2)--(2,2)--(2,1)--(3,1)--(3,0)--(4,0)--(4,4)--(0,4)--cycle;
\draw[entry](0,3)node{$1$}(1,3)node{$2$}(2,3)node{$3$}(3,3)node{$4$}(1,2)node{$5$}(2,2)node{$6$}(3,2)node{$8$}(2,1)node{$7$}(3,1)node{$9$}(3,0)node{$10$};
\end{scope}
\begin{scope}[xshift=6cm]
\draw[thinLine](3,1)--(4,1)(2,2)--(4,2)(1,3)--(4,3)(1,3)--(1,4)(2,2)--(2,4)(3,1)--(3,4);
\draw[thickLine](0,3)--(1,3)--(1,2)--(2,2)--(2,1)--(3,1)--(3,0)--(4,0)--(4,4)--(0,4)--cycle;
\draw[entry](0,3)node{$1$}(1,3)node{$2$}(2,3)node{$4$}(3,3)node{$5$}(1,2)node{$3$}(2,2)node{$6$}(3,2)node{$7$}(2,1)node{$8$}(3,1)node{$9$}(3,0)node{$10$};
\end{scope}
\end{scope}
\end{tikzpicture}
\caption{Partial promotion $\partial_2$ and its inverse, with the promotion paths highlighted.}
\label[fig]{promotion}
\end{figure}
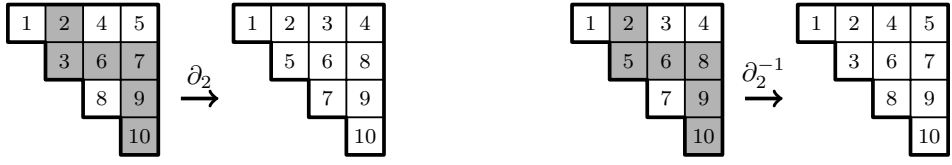

See \cref{promotion}.
In the case $k = 1$ partial promotion becomes Sch\"utzenberger's promotion \cite{Schu63}. See also \cite{St09}.

\begin{thm}\label[thm]{adjacencies}
Let $\lambda$ be a (possibly strict) partition.
Then the total number of horizontal adjacencies in column $c$ of (possibly shifted) SYT of shape $\lambda$ is equal to the number of (possibly shifted) SYT of shape $\lambda$ with largest entry in column $c+1$ or greater. 
\end{thm}
\begin{proof}
We prove this by constructing a bijection $\phi : S_c(\lambda) \rightarrow A_c(\lambda)$ between the set $S_c(\lambda)$ of all (shifted) SYT of shape $\lambda$ with largest entry in column $k\geq c+1$,
and the set $A_c(\lambda)$ of horizontal adjacencies $(\tab,u)$ in column $c$ of a (shifted) SYT $\tab$ of shape $\lambda$.
Consider a tableau $\tab$ of shape $\lambda$ and assume the largest entry $|\lambda|$ is in column $c+1$ or greater. Then the inverse promotion path of $\tau_1 \circ \tau_2 \circ \dots \circ \tau_{|\lambda|-1}$ has to cross from column $c+1$ to $c$ at some unique $k(\tab)$, that is, $k(\tab)$ is the first index at which the path of $\partial_{k(\tab)}^{-1} = \tau_{k(\tab)} \circ \dots \circ \tau_{|\lambda|-1}$ ends in column $c$ of $\tab$. By the definition of partial inverse promotion, the choice of $k(\tab)$ ensures that $k(\tab)$ and $k(\tab)+1$ are horizontally adjacent with $k(\tab)$ in some cell $u$ in column $c$ of $\partial_{k(\tab)}^{-1}(\tab)$. This is illustrated in \cref{adj_bij}.
Hence, by letting $\phi(\tab) = (\partial_{k(\tab)}^{-1}(\tab), u)$ we obtain a map $\phi : S_c(\lambda) \rightarrow A_c(\lambda)$.

\begin{figure}[htbp!]
\centering
\begin{tikzpicture}[scale=1.2]
\begin{scope}
\fill[fillGrey](1,3)--(2,3)--(2,2)--(3,2)--(3,0)--(4,0)--(4,3)--(3,3)--(3,4)--(1,4)--cycle;
\draw[thinLine](3,1)--(4,1)(2,2)--(4,2)(1,3)--(4,3)(1,3)--(1,4)(2,2)--(2,4)(3,1)--(3,4);
\draw[thickLine](0,3)--(1,3)--(1,2)--(2,2)--(2,1)--(3,1)--(3,0)--(4,0)--(4,4)--(0,4)--cycle;
\draw[entry](1,3)node{$k$};
\end{scope}
\draw[thickLine,->,xshift=46mm,yshift=2cm](0,0)--(1,0);
\begin{scope}[xshift=6cm]
\draw[thinLine](3,1)--(4,1)(2,2)--(4,2)(1,3)--(4,3)(1,3)--(1,4)(2,2)--(2,4)(3,1)--(3,4);
\draw[thickLine](0,3)--(1,3)--(1,2)--(2,2)--(2,1)--(3,1)--(3,0)--(4,0)--(4,4)--(0,4)--cycle;
\draw[entry,font=\scriptsize](1,3)node{$k$}(2,3)node{$k$+$1$};
\end{scope}
\end{tikzpicture}
\caption{The bijection in the proof of \cref{adjacencies} with $k = k(\tab)$ and the path of $\tau_{k(\tab)} \circ \dots \circ \tau_{|\lambda|-1}$ coloured.}
\label[fig]{adj_bij}
\end{figure}
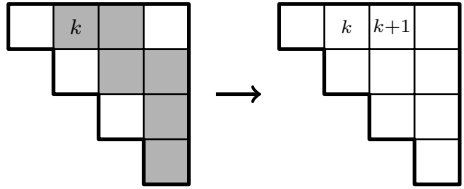

If $(S, u)$ is an adjacency in column $c$, the largest entry of $\partial_{S(u)}(S)$ has to be in column $c+1$ or greater as there is at least one entry in column $c+1$. Hence, by defining $\psi(S,u)=\partial_{S(u)}(S)$ we obtain a map $\psi:A_c(\lambda)\to S_c(\lambda)$.
Given the adjacency $(\partial_{k(\tab)}^{-1}(\tab), u)$ in column $c$, note that $\partial_{k(\tab)}^{-1}(\tab)(u) = k(\tab)$, so $\psi(\phi(\tab))= \partial_{k(\tab)}(\partial_{k(\tab)}^{-1}(\tab)) = \tab$. On the other hand, if $(S, u)$ is an adjacency in column $c$, then $\partial^{-1}_{S(u)}(\partial_{S(u)}(S)) = S$ so $k(\partial_{S(u)}(S)) = S(u)$. Hence, $\phi(\psi(S, u)) = \phi(\partial_{S(u)}(S)) = (\partial^{-1}_{S(u)}(\partial_{S(u)}(S)), u) = (S, u).$ Thus $\phi$ and $\psi$ are inverse bijections and the proof is complete. 
\end{proof}

Techniques similar to those used in the proof of \cref{adjacencies} also appear in \cite{STWW}.
\begin{cor}\label[cor]{expected_adjacencies}
The expected number of horizontal (resp.~vertical) adjacencies in column $c < n-1$ (resp.~row $r < n - 1$) of a uniformly random shifted staircase SYT is equal to 1.
\end{cor}
The previous corollary in turn leads to the following.
\begin{cor}
The expected number of adjacencies in a random 132-avoiding sorting network of length $N$ is $2(n-2)$.
\end{cor}

Compare this with the result of Schilling et al.~below.
\begin{thm}[{\cite[Thm.~1.3]{STWW}}]
The expected number of $i\ i + 1\ i $, $1 \leq i \leq n-1$, in a random 132-avoiding sorting network of length $N$ is 1.
\end{thm}

Some other corollaries of \cref{adjacencies} are listed next. 
\begin{cor}
For any strict partition $\lambda$, the expected number of horizontal adjacencies in column $c< \ell(\lambda)$ of a random shifted SYT of shape $\shyd{\lambda}$ is 1.
\end{cor}

A (shifted) Young diagram is called a \emph{(shifted) rectangle} if there is a cell that contains the largest entry in all standard tableaux of this shape. That is $\lambda=(\lambda_1,\dots,\lambda_{\ell})$ is a shifted rectangle if and only if
$\lambda_i=\lambda_1-i+1$ for all $1\le i\le \ell$.

\begin{cor}
Let $\yd{\lambda}$ (resp.~$\shyd{\lambda}$) be a (shifted) rectangle and fix a column $c$ of $\yd{\lambda}$ (resp.~$\shyd{\lambda}$).
Then the number of adjacencies in column $c$ is equal to $f_{\lambda}$ (resp.~$\shf{\lambda}$).
\end{cor}

\begin{cor}
Let $\Box=(n^m)$ be a rectangle.
Then
\begin{equation*}
\sum_{\yd{\mu}\subseteq\yd{\nu}\subseteq\yd{\Box}}f_{\mu}\cdot f_{\Box\setminus\nu}
=(n-1)f_{\Box},
\end{equation*}
where the sum is taken over all Young diagrams such that $\yd{\nu}=\yd{\mu}\cup\{u,\e u\}$ for some cell $u\in\yd{\Box}\setminus\yd{\mu}$.
Similarly,
\begin{equation}\label{E:shrectangle}
\sum_{\shyd{\mu}\subseteq\shyd{\nu}\subseteq\shyd{\Delta_n}}\shf{\mu}\cdot \shf{\Delta_n\setminus\nu}
=(n-2)\shf{\Delta_n},
\end{equation}
where the sum is taken over all shifted Young diagrams such that $\shyd{\nu}=\shyd{\mu}\cup\{u,\e u\}$ for some cell $u\in\shyd{\Delta_n}\setminus\shyd{\mu}$.
\end{cor}

Compare \eqref{E:shrectangle} to the identities
\begin{equation}\label{E:STWW}
\sum_{\shyd{\mu}\subseteq\shyd{\nu}\subseteq\shyd{\Delta_n}}\shf{\mu}\cdot\shf{\Delta_n\setminus\nu}
=\shf{\Delta_n},
\end{equation}
where the sum is taken over all shifted Young diagrams such that $\shyd{\nu}=\shyd{\mu}\cup\{u,\e u,\e\s u\}$ for some cell $u\in\shyd{\Delta_n}\setminus\shyd{\mu}$, and
\begin{equation}\label{E:trivial}
\sum_{\shyd{\mu}\subseteq\shyd{\nu}\subseteq\shyd{\Delta_n}}\shf{\mu}\cdot\shf{\Delta_n\setminus\nu}
=\binom{n}{2}\shf{\Delta_n},
\end{equation}
where the sum is taken over all shifted Young diagrams such that $\shyd{\nu}=\shyd{\mu}\cup\{u\}$ for some cell $u\in\shyd{\Delta_n}\setminus\shyd{\mu}$.
The first identity follows from the result in \cite{STWW} mentioned above and the second is trivially true.

\begin{prob} Can equations \eqref{E:shrectangle}, \eqref{E:STWW} and \eqref{E:trivial} be generalised?
\end{prob}

\smallskip 
 If $w=s_i v$ is a reduced word for the reverse permutation in $\S_n$, then so is $v s_{n-i-1}$. Thus we know that the probability 
of an adjacency is the same at every position in a random sorting network. This implies that the expected number of adjacencies before any position in a random sorting network grows linearly. 
Experiments suggest that the distribution of adjacencies converges also in probability to the uniform distribution, see \cref{distances}. Also the number of adjacencies in a 132-avoiding sorting network seems to grow in a nice 
manner.

\begin{figure}[t]
\begin{center}
\includegraphics[width=.45\textwidth]{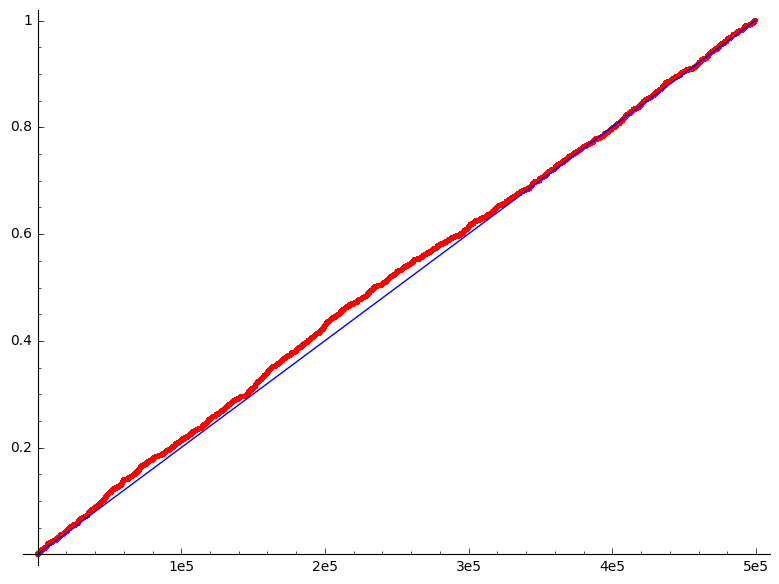}
\quad
\includegraphics[width=.45\textwidth]{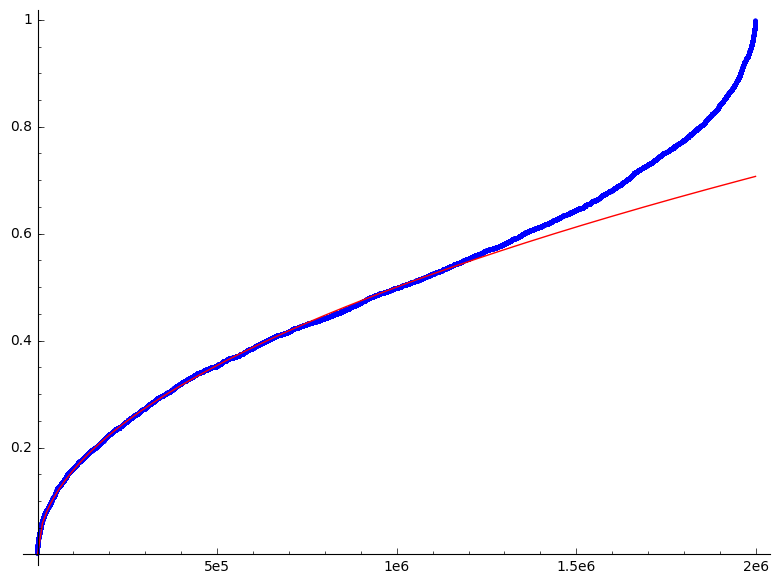}
\end{center}
\caption{The numbers of adjacencies in an initial segment of a random sorting network of size $n=1000$, (left) seems to grow linearly. The number of adjacencies in an initial segment of a random $132$-avoiding sorting network of size $n=2000$ (right) seems to grow like a square root for the first half.}
\label[fig]{distances}
\end{figure}

\begin{conj}
Let $Y_n(\alpha)$ and $X_n(\alpha)$ denote the number of adjacencies $1\leq k<\alpha\binom{n}{2}$ in a random sorting network and in a random $132$-avoiding sorting network of size $n$ respectively.

Then $Y_n(\alpha)/\E[Y_n(1)]$ converges in probability to $c\alpha$ for some constant $c$ and 
\begin{equation*}
\lim_{n\to\infty}\P\left(\max_{0 \leq \alpha \leq 1}\abs{X_n(\alpha)/(2(n-2))-g(\alpha)}>\epsilon\right)=0,
\end{equation*} where
\begin{equation*}
g(\alpha)=
\begin{cases}
\sqrt{\frac{\alpha}{2}}&\quad\text{if }\alpha\in[0,\frac{1}{2}],\\
1-\sqrt{\frac{1-\alpha}{2}}&\quad\text{if }\alpha\in[\frac{1}{2},1].
\end{cases}
\end{equation*}
\end{conj}

As an interesting problem for further study, we would also like to mention the behaviour of distances in value of adjacent elements in sorting networks, see \cref{adj_diff}. For example, if $\{w_k,w_{k+1}\}=\{j,j+2\}$ in a 132-avoiding sorting network, 
then $k$ and $k+1$ are diagonally adjacent in the corresponding shifted SYT $Q_w^{\rightarrow}$. 

\begin{figure}[ht]
\begin{center}
\includegraphics[width=.45\textwidth]{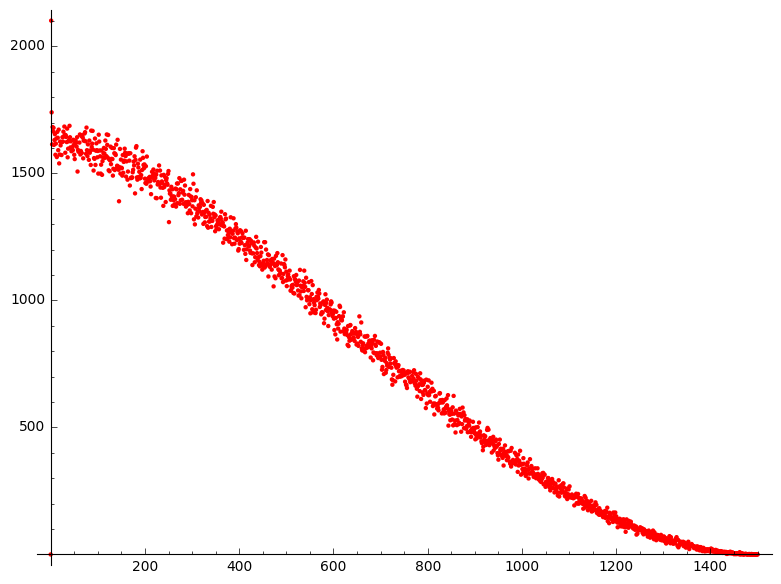}
\quad
\includegraphics[width=.45\textwidth]{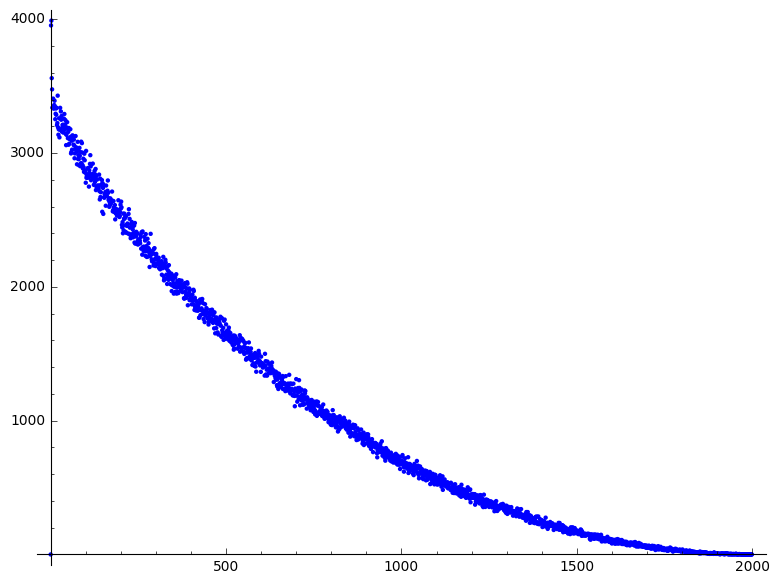}
\end{center}
\caption{The number of occurrences of distances between adjacent entries in a random sorting network of size $n=1500$ (left) and a random $132$-avoiding sorting network of size $n=2000$ (right).}
\label[fig]{adj_diff}
\end{figure}

\section*{Acknowledgements}
The authors thank Dan Romik for an informative correspondence.
This paper benefited from experimentation with \texttt{Sage}~\cite{sage} and its combinatorics features developed by the \texttt{Sage-Combinat} community~\cite{Sage-Combinat}.

\bibliographystyle{amsalpha}
\bibliography{132avoiding}
\end{document}